\documentclass[12pt,reqno]{amsart}
\usepackage{amssymb}
\usepackage{amsmath}
\usepackage{enumitem}
\usepackage{mathrsfs}
\usepackage[all]{xy}
\usepackage{hyperref}
\usepackage{marginnote}

\usepackage{tikz-cd}

\usepackage{stmaryrd}

\usepackage[margin=1.25in]{geometry}

\RequirePackage{mathrsfs} 

\newtheorem{theorem}{Theorem}[section]
\newtheorem{lemma}[theorem]{Lemma}
\newtheorem{proposition}[theorem]{Proposition}
\newtheorem{corollary}[theorem]{Corollary}
\newtheorem{conjecture}[theorem]{Conjecture}

\theoremstyle{definition}
\newtheorem*{ack}{Acknowledgements}
\newtheorem*{con}{Conventions}
\newtheorem{remark}[theorem]{Remark}
\newtheorem{example}[theorem]{Example}
\newtheorem{definition}[theorem]{Definition}

\numberwithin{equation}{section} \numberwithin{figure}{section}

\DeclareMathOperator{\Pic}{Pic}
\DeclareMathOperator{\PicS}{\mathbf{Pic}}

\DeclareMathOperator{\Spec}{Spec}

\DeclareMathOperator{\Br}{Br}

\let\Im\relax
\DeclareMathOperator{\Im}{Im}

\DeclareMathOperator{\Hilb}{Hilb}
\DeclareMathOperator{\chr}{char}

\newcommand{\GL}{\mathrm{GL}}
\newcommand{\PGL}{\mathrm{PGL}}

\newcommand{\Qbar}{\overline{\QQ}}

\newcommand\PP{\mathbb{P}}

\newcommand\ZZ{\mathbb{Z}}

\newcommand\QQ{\mathbb{Q}}

\newcommand\OO{\mathcal{O}}
\renewcommand\H{\mathrm{H}}

\renewcommand\P{\mathbb{P}}

\newcommand\Z{\mathbb{Z}}
\newcommand\N{\mathbb{N}}
\newcommand\Q{\mathbb{Q}}

\renewcommand\O{\mathcal{O}}

\usepackage{color}

\definecolor{orange}{rgb}{1,0.5,0}

\newcommand{\dan}[1]{{\color{blue} \sf $\clubsuit\clubsuit\clubsuit$ Dan: [#1]}}
\newcommand{\ari}[1]{{\color{red} \sf $\clubsuit\clubsuit\clubsuit$ Ari: [#1]}}

\author{Ariyan Javanpeykar}
\address{Ariyan Javanpeykar \\
	Institut f\"{u}r Mathematik\\
	Johannes Gutenberg-Universit\"{a}t Mainz\\
	Staudingerweg 9, 55099 Mainz\\
	Germany.}
\email{peykar@uni-mainz.de}

\author{Daniel Loughran}
\address{Daniel Loughran \\
Department of Mathematical Sciences \\
University of Bath \\
Claverton Down \\
Bath \\
BA2 7AY \\
UK.}
\urladdr{https://sites.google.com/site/danielloughran/}

\author{Siddharth Mathur}
\address{Siddharth Mathur \\
Mathematisches Institut\\
Heinrich-Heine-Universit\"at\\40204 D\"usseldorf, Germany.}
\urladdr{https://sites.google.com/view/sidmathur/home}


\title{Good reduction and cyclic covers}

\begin{document}
	
\begin{abstract}
	We prove finiteness results for sets of varieties over number fields
	with good reduction outside a given finite set of places
	using cyclic covers. We obtain a version of
	the Shafarevich conjecture for weighted projective surfaces, 
	double covers of abelian varieties, 
	and reduce the Shafarevich conjecture for hypersurfaces to the case
	of hypersurfaces of high dimension.	These are special cases of a general
	set-up for integral points on moduli stacks of cyclic covers, and our arithmetic
	results are achieved via a version of the Chevalley--Weil theorem for stacks.
\end{abstract}

	\maketitle
	\tableofcontents

	\thispagestyle{empty}

	\section{Introduction}

The Shafarevich conjecture \cite{Shaf1962}, as proved by Faltings \cite{Faltings2}, states that given a number field $K$, a finite set $S$ of places of $K$, and an integer $g\geq 2$, there are only finitely many smooth projective curves of genus $g$ over $K$ with good reduction outside of $S$. In this paper we prove analogues of this property for other classes of varieties using a new technique based upon cyclic covers and their moduli stacks.

As motivation for our method, consider hypersurfaces of degree $r$ in projective space. Our first result propagates the Shafarevich property from large dimension to small dimension. We achieve this using cyclic covers, specifically given a smooth hypersurface 
$$X: \quad f(x)=0 \quad \subset \PP^{n+1}_K,$$
of degree $r$, we want to consider the hypersurface
\begin{equation} \label{eqn:cyclic_cover}
	x_{n+2}^{r} = f(x) \quad \subset \PP^{n+2}_K
\end{equation}
and then interpolate between these two types of objects. However, one immediately runs into problems, since the hypersurface $X$ is not canonically defined by a fixed equation $f$: one can choose a different scalar multiple of $f$ without changing $X$. Over non-algebraically closed fields, these different choices yield non-isomorphic cyclic covers so there is no uniform way to associate a cyclic cover to a hypersurface. To overcome this ambiguity we consider all cyclic covers at once. This quickly leads one to study these objects in moduli and so we proceed by considering stacks of hypersurfaces and cyclic covers. The moduli stack framework is  well suited to Shafarevich-type problems as  smooth hypersurfaces of degree $r$ in $\mathbb{P}^{n+1}_{\OO_{K,S}}$ give rise to $\OO_{K,S}$-integral points on the moduli stack of hypersurfaces of degree $r$ in $\mathbb{P}^{n+1}_{\OO_{K,S}}$.

The new geometric input in the paper is a detailed study of these stacks. We create a new general framework that both clarifies and remedies the above ambiguity, and should be of independent interest. We work in the general setup of divisorial pairs  which builds on the case of cyclic covers of projective spaces treated in \cite{ArsieVistoli}. Our main geometric result is that the morphism of stacks given by associating to a cyclic cover its branch locus is \emph{proper \'etale} (Theorem \ref{thm:properetale}). In fact, in the case of hypersurfaces in $\mathbb{P}^n$, we even show that the associated stack of cyclic covers is a $\mu_r$-gerbe (Example \ref{ex:cyclicprojectivespace}). This allows us to perform a descent, \`{a} la Fermat, and get finiteness results through an application of a stacky version of the Chevalley--Weil theorem proved by the first-named authors in \cite{JLalg}. 

We now explain the applications of our method obtained in this paper.

\raggedbottom 

\subsection{Hypersurfaces in projective space}
Our first application concerns the following conjecture,
originally proposed in \cite[Conj.~1.4]{JL}.

	\begin{conjecture}[Shafarevich conjecture  for projective hypersurfaces]\label{conji} 
		Let $r\geq 2$ and $n\geq 1$. Then for all number fields $K$ and finite sets $S$ of finite places of $K$, the set of $\OO_{K,S}$-linear isomorphism classes of smooth hypersurfaces of degree $r$ in $\mathbb{P}^{n+1}_{\OO_{K,S}}$ is finite.
	\end{conjecture}
		
	Here by a \emph{linear isomorphism}, we mean an isomorphism induced by an automorphism of the ambient projective space.
For $n=1$ the conjecture follows from Faltings's finiteness theorem \cite{Faltings2} (one needs to be slightly careful with conics and cubics; see Proposition \ref{prop:plane_curves} for details). In higher dimensions, Conjecture \ref{conji} is known in the following cases: quadrics \cite[Prop.~5.1]{JL}, cubic surfaces \cite[Thm.~4.5]{Scholl}, quartic surfaces \cite[Cor.~1.3.2]{Andre}, sextic surfaces \cite[Thm.~1.3]{JLFano}, and cubic and quartic threefolds \cite[Thm.~1.1]{JL}. This conjecture has various equivalent reformulations in terms of existence of hypersurfaces with good reduction outside of a given finite set of finite places, in a closer vein to Shafarevich's original conjecture (see Proposition \ref{prop:Rewriting}).

Our first result says that to prove the conjecture, one may assume that $n$ is arbitrarily large with respect to $r$.

\begin{theorem}\label{thm:Shaf_via_Fano}
	Let $r \geq 2$ and $ n \geq 1$. Suppose that the Shafarevich conjecture holds for $n$-dimensional smooth hypersurfaces of degree $r$.
	Then it holds for $m$-dimensional smooth hypersurfaces of degree $r$
	for all $1\leq m \leq n$.
\end{theorem}

  We emphasize that, in Theorem \ref{thm:Shaf_via_Fano}, we assume, for \emph{every} number field $K$ and every finite set of finite places $S$ of $K$, the finiteness of $\OO_{K,S}$-linear isomorphism classes smooth hypersurfaces of degree $r$ in $\mathbb{P}^{n+1}_{\OO_{K,S}}$.  

In particular, it suffices to consider the case of Fano hypersurfaces (where $r \leq n$).

\begin{corollary}
	The Shafarevich conjecture for all smooth hypersurfaces follows from the Shafarevich conjecture for all smooth Fano hypersurfaces.
\end{corollary}

 This quite surprised the authors, as often in arithmetic geometry the most difficult case is that of varieties of general type.  
In fact, as one may take $n$ to be much larger than $r$, one need only consider unirational hypersurfaces by \cite{HMP}.

\subsection{Hypersurfaces in weighted projective space}

We expect that a similar statement to Conjecture \ref{conji} should hold for hypersurfaces in a weighted projective space. Our next theorem proves this for some such surfaces. 
	
\begin{theorem}\label{thm:weighted}  
		Let $r\geq 2$  be an integer   and let $A\subset \mathbb {C}$ be an integrally closed $\ZZ$-finitely generated subring such that $2\in A^\times$. Then the set of $A$-isomorphism classes of smooth surfaces of degree $2r$ in $\PP(1,1,1,r)_A$ is finite.
\end{theorem}

We prove this by using  moduli stacks of cyclic covers and performing a descent to appeal to Faltings' result for curves. For $r=2$, Theorem \ref{thm:weighted} recovers the Shafarevich conjecture for  del Pezzo surfaces of degree two due to Scholl \cite{Scholl}, and for $r=3$, the Shafarevich conjecture for polarized K3 surfaces of degree 2 due to Andr\'e \cite{Andre}. For any $r\geq 4$, Theorem \ref{thm:weighted} gives new cases of the Shafarevich conjecture for simply connected surfaces of general type which a priori have no relation to abelian varieties.
	
\subsection{Double covers of abelian varieties}
Our final application concerns double covers of abelian varieties. In its most precise form the result concerns the arithmetic hyperbolicity of a certain moduli stack parametrising such varieties, which in particular also gives finiteness results for double covers of torsors under abelian varieties. For the introduction, we content ourselves with the following slightly less precise statement. 

\begin{theorem}[Informal version of Theorem \ref{thm:Ggp}]  \label{mainthm:gentype}  
	Let $K$ be a number field and $S$ a finite set of finite places of $K$.
	Let $p,g \in \N$ with $g=2$ or $g \geq 4$.
	Then the set of isomorphism classes of smooth general type varieties $X$ over $K$ of dimension $g$, geometric genus $p$,  whose Albanese map over $\bar{K}$ is a double cover, and with good reduction outside of $S$, is finite.
\end{theorem}

We achieve this using recent work of Lawrence and Sawin on a version of the Shafarevich conjecture for hypersurfaces in abelian varieties \cite{lawrence2020shafarevich}, proved using the new  method of Lawrence and Venkatesh \cite{VL18}.

\subsection{Cyclic covering stacks}
Our main theorem on stacks of cyclic covers is Theorem \ref{thm:properetale}, which is the crucial new geometric input in the paper. We state here an informal version of this result and refer the reader to \S \ref{sec:stacks} for appropriate background and a complete statement.
We expect this result to be of independent interest and useful in other contexts where cyclic covers arise.

\begin{theorem}[Informal version of Theorem \ref{thm:properetale}] 
	Let $r \in \N$.
	Let $\mathcal{X}$ be the moduli stack over $\Z[1/r]$ whose objects are pairs $(X,D)$,
	where $X$ is a smooth proper variety with an ample divisor $D$. 
	Let $\mathcal{X}_r$ be the moduli stack 
whose objects are pairs $(X,D)\in \mathcal{X}$ together with a uniform cyclic covering of $X$ of degree $r$ ramified exactly along $D$. Then the natural morphism
	$$\mathcal{X}_r \to \mathcal{X},$$
	which associates to such a  cover its base variety and branch divisor,
	is proper \'etale (i.e. $\mathcal{X}_r$ is an \'etale gerbe over a finite \'etale cover of $\mathcal{X}$).
\end{theorem}

 \subsection*{Outline of the paper}
 In \S\ref{section:arithmetic_hyperbolicity} we recall some basic properties of arithmetic hyperbolicity from \cite{JLalg}. In \S \ref{sec:stacks} we then introduce the stacks we shall need and study their geometry in detail. In the remaining sections we use various special cases of these constructions to derive the applications stated in the introduction.

	\begin{ack}   We thank David Rydh for many useful comments and suggestions. We are grateful to Jack Hall and  Angelo Vistoli for helpful discussions, and  Brian Lawrence and Will Sawin for help with the proof of Theorem \ref{thm:sawinhyp}.   We thank the referee for helpful comments. The first named author gratefully acknowledges  support of   the IHES where part of this work was completed, as well as the University of Paris-Saclay for its hospitality.  The second-named author is supported by EPSRC grant EP/R021422/2. The third named author conducted this research in the framework of the research training group GRK 2240: Algebro-geometric Methods
in Algebra, Arithmetic and Topology, which is funded by the Deutsche Forschungsgemeinschaft. 
	\end{ack}

\begin{con}  We follow the conventions of \cite{JLalg} concerning varieties, and rings of $S$-integers in number fields, etc.

  Concerning gerbes, we follow the (standard) conventions of the stacks project \cite[Tag~06QB]{stacks-project}. For the definition of a $G$-gerbe, for a sheaf of groups $G$, see	 \cite[3.2]{EHKV}. For a gerbe \emph{banded} by an abelian group scheme $G$ see \cite[12.2.2]{OlssonBook}.

  If $\mathcal{F}$ is a locally free sheaf (of finite rank) on an algebraic stack $\mathcal{X}$ over a scheme $S$ and $r$ is an integer, we let $\mathcal{F}^r$ denote the $r$-th tensor power of $\mathcal{F}$. Moreover, we let $\mathbb{V}(\mathcal{F})$ denote the associated  vector bundle:
  \[v: \text{Spec}_{\mathcal{X}}(\text{Sym}^*\mathcal{F}^{\vee}) \to \mathcal{X}.\]  
 In particular, objects of $\mathbb{V}(\mathcal{F})(T)$ are pairs $x \in \mathcal{X}(T)$ and $s:\mathcal{O}_{\mathcal{X}_{T}} \to \mathcal{F}|_{\mathcal{X}_{T}}$. Finally, let $\mathbb{V}^{\circ}(\mathcal{F})$ denote the complement in $\mathbb{V}(\mathcal{F})$ of the zero section. Note that some authors will refer to the above vector bundle as $\mathbb{V}(\mathcal{F}^\vee)$ (see, for example, \cite[Tag~01M2]{stacks-project}).

Let $f: \mathcal{Z} \to \mathcal{S}$ and $g: \mathcal{X} \to \mathcal{Z}$ be morphisms of algebraic stacks. We denote by $f_*\mathcal{X}$ the Weil restriction of $\mathcal{X}$, viewed as a stack over $\mathcal{S}$. For any $\mathcal{S}$-scheme $T$ we have
\[f_*\mathcal{X}(T)=\text{Hom}_{\mathcal{Z}}(\mathcal{Z} \times_{\mathcal{S}} T, \mathcal{X}).\]
The stack $f_*\mathcal{X}$ is  algebraic if $f$ is proper and flat of finite presentation and $f \circ g$ is locally of finite presentation, quasi-separated, and has affine stabilizers \cite[Thm.~1.3]{hall2019coherent}. 
 
For an algebraic stack $\mathcal{X}$ we denote its Picard group by $\Pic \mathcal{X}$. For a representable morphism of algebraic stacks $f: \mathcal{X} \to \mathcal{Y}$, we denote by $\Pic_{\mathcal{X}/\mathcal{Y}}$ the sheafification of the Picard functor with respect to the fppf topology, and by $\PicS_{\mathcal{X}/\mathcal{Y}}=f_*(B\mathbb{G}_{m, \mathcal{X}})$, the Picard stack.

The Picard stack is algebraic when $f$ is a flat proper morphism of finite presentation. If, in addition, $f$ is cohomologically flat in dimension zero, the Picard functor is representable by an algebraic space (see \cite[Tags 0D2C and 0D04]{stacks-project}).  

For a stack $\mathcal{X}$, the universal object is the object of $\mathcal{X}$ corresponding to the identity morphism $\mathcal{X} \to \mathcal{X}$.
\end{con}

	\section{Arithmetic hyperbolicity and integral points on stacks}\label{section:arithmetic_hyperbolicity}
 
	\emph{Throughout this section $k$ is an algebraically closed field of characteristic zero}. We start with introducing the notion of arithmetic hyperbolicity for algebraic stacks following \cite[\S4]{JLalg}.

\begin{definition} \label{defn:arithmetic_hyperbolicity}
		A finitely presented algebraic stack  $X$ over  $k$ is \emph{arithmetically hyperbolic over $k$} if there exists a $\ZZ$-finitely generated subring    $A\subset k$ and a model $\mathcal X$ of $X$ over $A$  such that, for every $\ZZ$-finitely generated subring  $A\subset A'\subset k$, the set 
		$$\Im[ \pi_0(\mathcal X(A'))~\to~\pi_0(\mathcal X(k))] \quad \text{is finite}.  $$
\end{definition}
	
Note that $\mathcal{X}(A)$ is a groupoid, so we consider finiteness properties related to $\pi_0(\mathcal{X}(A))$, which denotes the set of isomorphism classes of objects in $\mathcal{X}(A)$.
We emphasise that arithmetic hyperbolicity is a statement about the image of the integral
points inside the points over the algebraic closure. The advantage of this definition
is that one may pass to a finite field extension which simplifies many
arguments by trivialising any Galois-theoretic data. To get a result about finiteness
of integral points over the ground field, one uses the twisting lemma below. It
requires one to establish separation properties of the stack, but crucially these can be achieved using geometric arguments about the objects they parametrize
which, a priori, have nothing to do with arithmetic. The ``finite diagonal'' hypothesis in the statement holds for separated Deligne--Mumford stacks, for example.

\begin{theorem}[Twisting lemma, {\cite[Thm.~4.23]{JLalg}}]\label{thm:twists}
	Let $X$ be a finite type algebraic stack   over $k$ with finite diagonal. The following statements are equivalent.
	\begin{enumerate}
		\item The stack $X$ is arithmetically hyperbolic over $k$. 
		\item For all $\ZZ$-finitely generated integrally closed subrings $A\subset k$  and all     models $\mathcal X\to \Spec A$ for $X$ over $A$ with finite diagonal, the  set  $ \pi_0(\mathcal X(A)) $ of isomorphism classes of $A$-integral points on $\mathcal{X}$ is finite.
	\end{enumerate}
\end{theorem}

The classical Chevalley--Weil theorem says that arithmetic hyperbolicity descends along finite \'etale morphisms of integral varieties over $k$. The following result provides a stacky extension of this result.

	\begin{theorem}[Stacky Chevalley--Weil, {\cite[Thm~5.1]{JLalg}}]\label{thm:chev_weil}
		Let $f:X\to Y$ be a surjective proper \'etale morphism of finitely presented algebraic stacks over $k$. Then $X$ is arithmetically hyperbolic if and only if $Y$ is arithmetically hyperbolic.
	\end{theorem}

We will make use of the following from \cite[Prop.~4.16]{JLalg}.

\begin{lemma}\label{lem:qf}
		Let $  Y\to Z$ be a quasi-finite morphism of finitely presented algebraic stacks over $k$. If $Z$ is arithmetically hyperbolic, then 
		$Y$ is arithmetically hyperbolic. 
\end{lemma}  

We also require the following finiteness statement in cohomology \cite[Prop.~5.1]{GilleMoretBailly}.

\begin{lemma} \label{lem:GMB}
	Let $K$ be a number field, $S$ a finite set of finite places of $K$ and $G$ a finite
	type affine group scheme over $\O_{K,S}$. Then $\H^1(\O_{K,S}, G)$ is finite.
\end{lemma}

\section{Moduli stacks of cyclic covers} \label{sec:stacks}

We now study moduli stacks of cyclic covers in a general setting; the special case of cyclic covers of projective space was introduced in \cite{ArsieVistoli}.  Although we use the conventions of $loc.\ cit.$, we develop a theoretical framework in a greater generality (see Definition \ref{def:universalbranch}).

Our main result is that these moduli stacks are proper \'etale over a moduli stack of divisorial pairs (Theorem \ref{thm:properetale}). Since we work in a great generality, we require a lot of set-up and background, which we now introduce. To help the reader, we have various running examples throughout.

\subsection{Polarizing line bundles}

We introduce the mother of all moduli stacks (of smooth varieties) $\mathcal{P}ol^{\mathcal{L}}$ and its variant $\mathcal{P}ol^\Lambda$. Many of the moduli stacks we consider in this paper will be realized as a locally closed locus in $\mathcal{P}ol^\Lambda$ (Examples  \ref{ex:BSpol}, \ref{ex:k3s}). Others can easily be realized as a  vector bundle over $\mathcal{P}ol^{\mathcal{L}}$ (Examples \ref{ex:BS}, \ref{defn:ars}).

Recall that there is an algebraic stack, locally of finite presentation over $\Spec \mathbb{Z}$, parametrizing objects $(f: X \to S, L)$ where $f$ is a proper, flat morphism of finite presentation and $L$ is a relatively ample line bundle on $X/S$ (see \cite[Tags 0D1M, 0D4X, 0DPU]{stacks-project}). The locus of such pairs where $f$ is smooth with geometrically connected fibres is open (see, for example, \cite[Appendix E.1 (12)]{torstenalg}), and we denote the resulting open substack by $\mathcal{P}ol^{\mathcal{L}}$. As explained in \cite[4.2]{abramovich2011stable}, there is an associated algebraic stack parametrizing pairs $(f: X \to S, \lambda)$ where $f$ is as above and $\lambda$ is a relatively ample section of the Picard sheaf $\Pic_{X/S}$; we denote it by $\mathcal{P}ol^\Lambda$. In fact, there is a morphism $\mathcal{P}ol^{\mathcal{L}} \to \mathcal{P}ol^\Lambda$ sending $(f: X \to S, L) \mapsto (f: X \to S, [L])$ where $[L] \in \text{Pic}_{X/S}(S)$ (see Remark \ref{rem:correslb}).

\begin{definition} \label{def:canpol} 
An algebraic stack $\mathcal{X}^{\mathcal{L}}$ is said to be a moduli stack of \emph{varieties with a polarizing line bundle} if there is an immersion $\mathcal{X}^{\mathcal{L}} \hookrightarrow \mathcal{P}ol^{\mathcal{L}}$.
An algebraic stack $\mathcal{X}^\Lambda$ is said to be a moduli stack of \emph{polarized varieties} if there is an immersion $\mathcal{X}^\Lambda \hookrightarrow \mathcal{P}ol^\Lambda$. 
\end{definition}

\begin{example} \label{ex:BSpol} Let $\mathcal{B}_{(r;n)}$ denote the locus of $\mathcal{P}ol^{\mathcal{L}}$ where the universal object $(u: \mathcal{U} \to \mathcal{P}ol^{\mathcal{L}}, L)$  is a relative Brauer-Severi scheme of dimension $n+1$ such that $L$ has degree $r$ on the geometric fibers of $u$. This is open since it coincides with the locus where the relative tangent bundle $\mathcal{T}_{X/S}$ is a relatively ample vector bundle of rank $n+1$ and $u_*L$ is locally free of rank ${n+1+r}\choose{r}$, by a theorem of Mori \cite[Thm. 8]{Mori79} (see also \cite[Prop.~4.4]{hartshorne1966ample}). 
\end{example}

\begin{example} \label{ex:k3s} Let $\mathcal{F}_{2}$ denote the category of \emph{degree two polarized K3 surfaces}. More precisely, the groupoid $\mathcal{F}_2(T)$ consists of pairs $(X \to T, \lambda)$ where $X \to T$ is a proper smooth morphism whose  fibres are K3 surfaces and $\lambda \in \Pic_{X/T}(T)$ is locally representable by an ample line bundle $L$ satisfying $L^2=2$. Morphisms are cartesian diagrams that preserve the section of $\Pic_{X/T}(T)$. This is an example of a moduli stack of polarized varieties (see \cite[4.2.1 and 4.3.3]{Rizov}). Moduli stacks of K3 surfaces of higher degree are also of this form (see \cite{Rizov}).\end{example}

\begin{remark} \label{rem:correslb} 
Given a moduli stack of varieties with a polarizing line bundle $\mathcal{X}^{\mathcal{L}}$, we may associate to it a moduli stack of polarized varieties $\mathcal{X}^\Lambda$ (see \cite[4.2]{abramovich2011stable}). Indeed, define $\mathcal{X}^\Lambda$ to be the rigidification $\mathcal{X}^{\mathcal{L}}\fatslash \text{ } \mathbb{G}_m$ where $\mathbb{G}_m$ acts by scalar multiplication on the polarizing line bundles. This yields a $\mathbb{G}_m$-gerbe morphism $\mathcal{X}^{\mathcal{L}} \to \mathcal{X}^\Lambda$ sending $(X \to S, L) \mapsto (X \to S, [L])$, where $[L]$ denotes the class of $L$ in $\Pic_{X/S}$.

Conversely, given a moduli stack of polarized varieties $\mathcal{X}^{\Lambda}$, the universal object $(\mathcal{U} \to \mathcal{X}^{\Lambda}, \lambda)$ induces a morphism $\lambda: \mathcal{X}^{\Lambda} \to \Pic_{\mathcal{U}/\mathcal{X}}$. We set 
\[\mathcal{X}^{\mathcal{L}}=\mathcal{X}^{\Lambda} \times_{\Pic_{\mathcal{U}/\mathcal{X}^{\Lambda}}} \textbf{Pic}_{\mathcal{U}/\mathcal{X}^{\Lambda}},\]
where $\textbf{Pic}_{\mathcal{U}/\mathcal{X}^{\Lambda}} \to \Pic_{\mathcal{U}/\mathcal{X}^{\Lambda}}$ is the natural $\mathbb{G}_m$-gerbe (see \cite[Tags 0DME and 0DNH]{stacks-project}).  This yields a moduli stack of varieties with polarizing line bundles $\mathcal{X}^{\mathcal{L}}$. These two processes yield a correspondence between moduli stacks of polarized varieties and moduli stacks of varieties with polarizing line bundle. 
\end{remark}

A polarized variety $(X \to S, \lambda)$ does not canonically induce a map to projective space in general, but as the following  shows, there is a canonical map to a Brauer-Severi scheme. This will allow us to interpret certain moduli stacks of polarized varieties as moduli stacks of ramified coverings (e.g.~Proposition \ref{prop:k3s}), and vice-versa.

\begin{proposition} \label{prop:factorBS} 
Let $\pi: \mathcal{X}^{\mathcal{L}} \to \mathcal{X}^\Lambda$ be the $\mathbb{G}_m$-gerbe map in Remark \ref{rem:correslb} and suppose $(u: \mathcal{U} \to \mathcal{X}^{\mathcal{L}}, L)$ and $(v: \mathcal{V} \to \mathcal{X}^{\Lambda}, \lambda)$ are the corresponding universal objects. Assume that $u_*L$ is locally free of rank $r+1$, compatible with arbitrary base change, and the natural map $u^*u_*L \to L$ is surjective. Then there is a factorization
\[\mathcal{V} \to \mathcal{P} \to \mathcal{X}^{\Lambda}\]
where $\mathcal{P}$ is a relative Brauer-Severi scheme of dimension $r$ over $\mathcal{X}^{\Lambda}$. Moreover, the restriction of $[\mathcal{O}_{\mathcal{P}}(1)] \in \Pic_{\mathcal{P}/\mathcal{X}^{\Lambda}}(\mathcal{X}^{\Lambda})$ to $\mathrm{Pic}_{\mathcal{V}/\mathcal{X}^{\Lambda}}(\mathcal{X}^{\Lambda})$ is $\lambda$. 
\end{proposition}

\begin{proof} 
Note that $\mathcal{U} \cong \mathcal{V} \times_{\mathcal{X}^{\Lambda}} \mathcal{X}^{\mathcal{L}}$ and that $L$ is $1$-twisted with respect to the $\mathbb{G}_m$-banding on $\mathcal{X}^{\mathcal{L}}$ (see \cite[3.1.1.1]{lieblich2008twisted}). By hypothesis we obtain a morphism
\[\mathcal{U} \to \mathbb{P}(u_*L) \to \mathcal{X}^{\mathcal{L}}\]
The rigidification of these stacks with respect to $\mathbb{G}_m$ yields the desired result. The last statement follows because $\mathcal{O}_{ \mathbb{P}(u_*L)}(1)|_{\mathcal{U}}=L$. 
\end{proof}

\subsection{Stacks of divisorial pairs}

We now introduce algebraic stacks which parametrize pairs $(X,H)$ where $X$ is a variety and $H \subset X$ is an ample Cartier divisor. 

\begin{definition} \label{def:pairs} 
Let $\mathcal{X}$ denote a moduli stack of varieties with polarizing line bundle and let $(u: \mathcal{U} \to \mathcal{X}, L)$ denote its universal object. Suppose that $u_*L$ is locally free and compatible with arbitrary base change on $\mathcal{X}$. An open subset of $\mathbb{V}^{\circ}(u_*L)$ is said to be a moduli stack of \emph{divisorial pairs}.
\end{definition}

\begin{remark} \label{rem:explain2} 
Unravelling the definition, the objects of a moduli stack of divisorial pairs are triples $(f: Y \to S, L, \sigma: \mathcal{O}_Y \to L)$, where $f$ is a smooth proper morphism of finite presentation with geometrically connected fibers, $L$ is a relatively ample line bundle on $Y/S$, $f_*L$ commutes with base change on $S$, and $\sigma$ is injective and remains so on every fibre of $f$. A morphism between triples $(Y' \to S', L', \sigma': \mathcal{O}_{Y'} \to L') \to (Y \to S, L, \sigma: \mathcal{O}_Y \to L)$ over $S' \to S$ is a pair $(f,\phi)$ where

    \begin{center}
    \begin{tikzcd}
  Y' \arrow[r, "f"] \arrow[d] & Y \arrow[d] \\
 S' \arrow[r]
& S 
\end{tikzcd}
\end{center}
is cartesian and $\phi: f^*L \to L'$ is an isomorphism sending $f^*\sigma$ to $\sigma'$. By considering the vanishing of the section $\sigma$, we see that this is equivalent to the category parametrizing objects $(f: Y \to S, i: H \hookrightarrow Y)$ where $f$ is a proper smooth morphism of finite presentation with geometrically connected fibers, $i$ is a locally principal closed immersion which is an ample Cartier divisor on every fibre $Y_s$ (equivalently, $H$ is a relatively ample Cartier divisor which is flat over $S$ by \cite[Tag 062Y]{stacks-project}), and where $f_*\mathcal{O}_Y(H)$ commutes with base change on $T$. A morphism $(f': Y' \to S', i': H' \to Y') \to (f: Y \to S, i: H \to Y)$ is a pair of morphisms $(f,g)$ making all squares in the following diagram cartesian. 

    \begin{center}
    \begin{tikzcd}
 H' \arrow[d,"g"] \arrow[r, "i"] & Y' \arrow[r] \arrow[d, "f"] & S' \arrow[d] \\
 H \arrow[r, "i'"] & Y \arrow[r]
& S
\end{tikzcd}
\end{center}
 \end{remark}

\begin{example} \label{ex:BS} Let $\mathcal{C}_{(r;n)}$ denote the stack parametrizing pairs $(f: P \to S, j: H \hookrightarrow P)$ where $f$ is a Brauer-Severi scheme of dimension $n+1$, $j$ is a closed immersion which is the inclusion of a Cartier divisor of degree $r$ on every fibre, and $H$ is smooth over $S$. Then $\mathcal{C}_{(r;n)}$ is a moduli stack of divisorial pairs. Indeed, note that $\mathcal{C}_{(r;n)}$ can be identified with an open subset of $\mathbb{V}^{\circ}(u_*L)$ over $\mathcal{B}_{(r;n)}$ (see Example \ref{ex:BSpol}). In fact, this is a well-studied object and we give another description in \S \ref{section:pres}. \end{example}

\subsection{Uniform cyclic covers}
Branched covers of varieties are ubiquitous in algebraic geometry. We consider the distinguished subclass consisting of uniform cyclic covers. Their rigid structure allows one to interpolate between the branch locus and the cover in a way we will make precise. We will define the stack of such covers, building on the case of cyclic covers of projective spaces treated in \cite{ArsieVistoli}. We begin with \cite[Def.~2.1]{ArsieVistoli}.

\begin{definition} \label{def:cyclic_cover}
	Let $Y$ be a scheme. A \emph{uniform cyclic cover of degree $r$} of $Y$ consists of a morphism of schemes $f : X \to Y$ together with an action of the group
scheme $\mu_r$ on $X$, such that for each point $y \in Y$, there is an affine neighbourhood $V = \Spec R$ of $y$, together with an element $h \in R$ that is not a zero divisor, and
an isomorphism of $V$-schemes $f^{-1}(V) \cong \Spec R[x]/(x^r - h)$ which is $\mu_r$-equivariant,
when the right hand side is given the obvious action.
\end{definition}

Throughout the paper, for brevity we often simply write ``cyclic cover'' instead of ``uniform cyclic cover''.

\begin{remark} \label{rem:linebundle} To  a cyclic cover $f: X \to Y$ is associated a line bundle $\mathcal{L}$, given by the subsheaf of $f_{\ast}\mathcal{O}_X$ on which $\mu_r$ acts via scalar multiplication. Indeed, the $\mu_r$-action on $X$ over $Y$ is equivalent to a $\mathbb{Z}/r\mathbb{Z}$-grading 
\[f_{\ast}\mathcal{O}_X=\mathcal{O}_Y \oplus \mathcal{L} \oplus \mathcal{L}^{2} \oplus \dots \oplus \mathcal{L}^{r-1} \]  on the algebra $f_{\ast}\mathcal{O}_X$ 
via the eigensheaf decomposition.
In particular, there is an injective map $\mathcal{L}^{r} \to \mathcal{O}_Y$. The \emph{branch divisor} of the cover $f$ is defined to be the Cartier divisor associated to the sheaf of ideals given by the image of $\mathcal{L}^{r}$ in $\OO_Y$. See \cite[\S 2]{ArsieVistoli} for a thorough treatment of the structure of such covers.
\end{remark}

\begin{remark} \label{rem:rth_root}
	The data of a cyclic cover over a scheme $Y$  branched over a Cartier divisor $H \subset Y$ is the same as the data of an $r$th root of the line bundle $\O(H)$ and a chosen section $s \in \O(H)$ defining $H$. Indeed, if $\mathcal{L}$ is such a root, we give the module
	\[\mathcal{O}_Y \oplus \mathcal{L}^{-1} \oplus \dots \mathcal{L}^{1-r}\]
	a $\mathbb{Z}/r\mathbb{Z}$-graded $\mathcal{O}_Y$-algebra structure $\mathcal{A}_s$ by defining multiplication using the morphism $s^{\vee}: \mathcal{L}^{-r}=\mathcal{O}(-H) \to \mathcal{O}_Y$. Then $\Spec_Y \mathcal{A}_s \to Y$ is a cyclic cover branched over $H$. If $s'$ is another section defining $H$, then one can check that $\mathcal{A}_s \simeq \mathcal{A}_{s'}$ as $\mathbb{Z}/r\mathbb{Z}$-graded $\mathcal{O}_Y$-algebras exactly when $\frac{s}{s'}$ admits an $r$th root in $\mathcal{O}_Y^{\times}$. Thus, different choices of $s$ can yield non-isomorphic cyclic covers which share the same branching data. Although this ambiguity complicates the correspondence between cyclic covers and their branch loci, we shall clarify this relationship by studying their moduli.
\end{remark}

\subsection{Relative uniform cyclic covers}

\begin{definition} \label{def:relative} Fix a scheme $S$ and let $Y$ be a flat $S$-scheme of finite presentation. A \emph{relative uniform cyclic cover} of degree $r$ of $Y$ is a uniform cyclic cover $f: X \to Y$ of degree $r$ such that the branch divisor of $f$ is flat over $S$. A \emph{morphism} of relative uniform cyclic covers $X \to Y$ and $X' \to Y$ is a $\mu_r$-equivariant morphism over $Y$.
\end{definition}

\begin{remark} \label{rem:flatness}
 By \cite[Tag 062Y]{stacks-project}, $S$-flatness of the branch divisor is equivalent to the injection $\mathcal{L}^{r} \to \OO_Y$ remaining injective upon restriction to every fibre $Y_s$. In particular, this implies forming the branch divisor (as defined in \ref{rem:linebundle}) of a relative uniform cyclic cover $X \to Y \to S$ is compatible with arbitrary base change on $S$. This implies that relative uniform cyclic covers are stable under arbitrary base change. As such, a relative uniform cyclic cover over $S$ can be thought of as a family of uniform cyclic covers (see Definition \ref{def:cyclic_cover}) parametrized by $S$. 
 \end{remark} 

\begin{example}  \label{ex:branch_degree}
Given a scheme $S$, a Brauer-Severi scheme $P\to S$, and a relative uniform cyclic cover $f:X\to P$ of degree $r$ with associated line bundle $\mathcal{L}$, the branch divisor has degree $rd$ where $d$ is the degree of $\mathcal{L}^{-1}$ on every geometric fibre of $P\to S$.
\end{example}

The following proposition is well-known. It is stated in \cite[Prop.~2.5]{ArsieVistoli} but is missing the hypothesis on $r$ being invertible on $S$; we give a proof with the corrected hypotheses for completeness.

\begin{proposition} \label{rem:smoothness}
Let $Y \to S$ be a smooth morphism of schemes and $g: X \to Y$ be a relative uniform cyclic cover of degree $r$ over $S$ with $r \in \mathcal{O}_S^{\times}$. Then $X$ is smooth over $S$ if and only if the branch divisor $D$ of $g$ is smooth over $S$.
\end{proposition}
\begin{proof}   Since $X$ and the branch divisor of $g$ are $S$-flat, we may assume $S$ is a geometric point. Moreover, we work locally on $Y$ so we will assume that $X=\Spec R[x]/(x^r-f)$ where $Y=\Spec R$ and $f \in R$ is not a zero divisor.

Away from $V(f)=D$, the morphism $g$ is \'etale because $r$ is invertible, hence smoothness of $X$ is automatic here. Thus, we only have to consider points lying over $D$.  Moreover, we may pass to the complete local ring of a closed point on $Y$ lying along $D$ and so by the Cohen structure theorem (see \cite[Tag 0C0S]{stacks-project}) and the smoothness of $Y$, we can write $g$ as the natural morphism $X=\Spec R[x]/(x^r-f) \to \Spec R$ where $f \in R=k[[y_1,...,y_n]]$. Now, consider the partial derivatives of the equation defining $X$: $\partial_x(x^r-f)=rx^{r-1}$ and $\partial_{y_i}(x^r-f)=\partial_{y_i}(-f)$. Then, these $n+1$ equations have a common zero $(r_1,...,r_n,b)$ which lies along $V(x^r-f)$ if and only if $b=0$ (since $r \in R^{\times}$), $f(r_1,...,r_n)=0$, and $\partial_{y_i}(f)(r_1,...,r_n)=0$ for every $i$. In other words, $X$ is not smooth if and only if $V(f)=D$ is not smooth. \end{proof}

\begin{remark} \label{rem:alternatecyclic} There is another well-known description of a relative uniform cyclic cover $X \stackrel{f}{\to} Y \to T$. The graded $\mathcal{O}_Y$-algebra structure on $f_*\mathcal{O}_X$ allows us to extract a line bundle $\mathcal{L}$ (as in Remark \ref{rem:linebundle}) and an injection $s: \mathcal{O}_Y \to \mathcal{L}^{-r}$ which remains injective on every fibre $Y_t$. In fact, this data recovers $X \to Y \to T$. Indeed, let $\mathcal{H}(Y/T,r)$ denote the groupoid of relative uniform cyclic covers of $Y$ of degree $r$ (with $\mu_r$-equivariant isomorphisms over $T$). Moreover, let $\mathcal{H}'(Y/T,r)$ denote the groupoid of pairs $(L, s: \mathcal{O}_Y \to L^{r})$ (where $s$ is injective and remains so on all geometric fibres $Y_t$) and where morphisms are isomorphisms $\alpha: L \to L'$ so that $\alpha^{\otimes r}$ is compatible with $s$ and $s'$. By \cite[Prop.~2.2]{ArsieVistoli}, there is a natural equivalence of categories $\mathcal{H}(Y/T,r) \simeq \mathcal{H}'(Y/T,r)$. \end{remark}

Thus, the data of a relative uniform cyclic cover $X \to Y \to T$ (of degree $r$) is equivalent to the data $(\pi: Y \to T, L, s: \mathcal{O}_Y \to L^{r})$ where $Y \to T$ is the base family, $L$ is a line bundle on $Y$ and $s$ is a global section of $L^{r}$ which doesn't vanish on any fibre $Y_t$. Equivalently, $s$ is a global section of $\pi_*L^{r}$ which does not meet the zero section. This motivates the following definition.

\subsection{Stacks of  cyclic covers}

\begin{definition} \label{def:stackcyclic} Let $\mathcal{X}^{\mathcal{L}}$ be a moduli stack of varieties with a polarizing line bundle. Let $(u: \mathcal{U} \to \mathcal{X}^{\mathcal{L}}, L)$ be the universal object and suppose that $u_* L^r$ is locally free and compatible with arbitrary base change on $\mathcal{X}^{\mathcal{L}}$. Then an open subset of $\mathbb{V}^{\circ}(u_*L^{r})$ is said to be a moduli stack of \emph{cyclic covers of degree $r$}. 
\end{definition}

\begin{remark} \label{rem:explain} Unravelling the definition, one sees that such a stack parametrizes triples $(f: Y \to S, L, s: \mathcal{O}_X \to L^r)$ where $(Y \to S, L)$ satisfies the conditions in Definition \ref{def:canpol}, $f_*L^r$ is compatible with base change, and $s$ is injective on all the fibres. A morphism 
\[(Y' \to S', L', s': \mathcal{O}_X \to L'^r) \to (Y \to S, L, s: \mathcal{O}_X \to L^r)\]
is a pair $(f, \phi)$ where 

    \begin{center}
    \begin{tikzcd}
  Y' \arrow[r, "f"] \arrow[d] & Y \arrow[d] \\
 S' \arrow[r]
& S 
\end{tikzcd}
\end{center}

\noindent is cartesian and $\phi: f^*L \to L'$ is an isomorphism such that $\phi^{\otimes r}(f^*s)=s'$.

Equivalently, by Remark \ref{rem:alternatecyclic} (see also \cite[Remark 3.3]{ArsieVistoli}), such a stack parametrizes pairs $(f: Y \to S, p: X \to Y)$ where $f$ is a smooth proper morphism with geometrically connected fibers and $p$ is a relative uniform cyclic cover of degree $r$. Morphisms in the stack are morphisms of relative uniform cyclic covers (see Definition \ref{def:relative}). 
\end{remark}

\begin{example}\label{defn:ars}
 For $n$, $r$, and $d$ positive integers, we let $\mathcal{H}(n,r,d)$ be the stack of relative uniform cyclic covers of degree $r$ branched over hypersurfaces of degree $rd$ in $\mathbb{P}^n$ over $\mathbb{Z}$, as defined in     \cite[\S 3]{ArsieVistoli}.    
	We let $\mathcal{H}^{\textrm{sm}}(n,r,d)$ denote the stack of   such cyclic covers that are smooth. These are both examples of moduli stacks of cyclic covers (of degree $r$). Indeed, by Example \ref{ex:BSpol} there is a moduli stack of polarizing varieties $\mathcal{B}_{(d;n-1)}$ parametrizing pairs $(f: X \to S, L)$ where $f$ is a Brauer-Severi scheme of dimension $n$ and $L$ has degree $d$ on the fibres. If $(\pi: \mathcal{X} \to \mathcal{B}_{(d;n-1)}, L)$ is the universal object, then $\mathcal{H}(n,r,d)=\mathbb{V}^{\circ}(\pi_*L^{r})$. Moreover, $\mathcal{H}^{sm}(n,r,d)$ is an algebraic stack with finite diagonal over $\Spec \mathbb{Z}$ whenever $rd \geq 3$ (see \cite[Remark 4.3]{ArsieVistoli}).  
	\end{example}

The following new definition is key to the paper and introduces the universal stack of cyclic covers over a moduli stack of divisorial pairs.

\begin{definition} \label{def:universalbranch}
Let $\mathcal{Y}$ be a moduli stack of varieties with polarizing line bundle with  universal object $(\pi: \mathcal{U} \to \mathcal{Y}, \mathcal{L})$. Suppose $\mathcal{X} \subset \mathbb{V}^{\circ}(\pi_*\mathcal{L})$ is a moduli stack of divisorial pairs. The algebraic stack $\mathcal{X}_r$ formed by the Cartesian diagram
    \begin{equation} \label{diagrambranch}
    \begin{tikzcd} 
 \mathcal{X}_r \arrow[r, "\text{open}"] \arrow[d] & \mathbb{V}^{\circ}(\pi_*\mathcal{L}|_{\mathcal{Y}_r}) \arrow[r] \arrow[d] & \mathcal{Y}_r \arrow[r] \arrow[d] 
& \PicS_{\mathcal{U}/\mathcal{Y}} \arrow[d, "(\cdot)^{r}"] \\
\mathcal{X} \arrow [r,"\text{open}"] & \mathbb{V}^{\circ}(\pi_*\mathcal{L}) \arrow [r] & \mathcal{Y} \arrow[r, "\mathcal{L}"]
& \PicS_{\mathcal{U}/\mathcal{Y}}.
\end{tikzcd}
\end{equation}
is called the \emph{universal stack of cyclic covers (of degree $r$) over $\mathcal{X}$}. 
\end{definition} 

In the next theorem, we verify that this is indeed a moduli stack of cyclic covers, and moreover prove that the morphism $\mathcal{X}_r \to \mathcal{X}$ is proper \'etale. To guarantee non-emptiness of $\mathcal{X}_r$, we impose the following condition: We say that a moduli stack of divisorial pairs $\mathcal{X}$ is \emph{$r$-divisible} if for all objects $(X,D)$ in $\mathcal{X}$ over a field, the line bundle $\O_X(D)$ is divisible by $r$ over the algebraic closure.

\begin{theorem} \label{thm:properetale} 
Let $\mathcal{X}$ be a moduli stack of divisorial pairs over a scheme $S$.
\begin{enumerate}
	\item The universal stack of cyclic covers $\mathcal{X}_r$ is a moduli stack of cyclic covers.
	\item The morphism $\mathcal{X}_r \to \mathcal{X}$ is proper and quasi-finite.
	\item If $r \in \mathcal{O}_S^{\times}$, then $\mathcal{X}_r \to \mathcal{X}$ is \'etale.
	\item If $\mathcal{X}$ is $r$-divisible, then $\mathcal{X}_r \to \mathcal{X}$  is  surjective.
\end{enumerate}
\end{theorem}

\begin{proof}
Let $(\pi: \mathcal{U} \to \mathcal{Y}, \mathcal{L})$ be the universal object of $\mathcal{Y}$ (see Definition \ref{def:universalbranch}).
 In   diagram \eqref{diagrambranch}, the stack $\mathcal{Y}_r$ is equivalent to a moduli stack of varieties with a polarizing line bundle and its universal object is of the form $(\mathcal{U}|_{\mathcal{Y}_r} \to \mathcal{Y}_r, \mathcal{M})$ with the property that $\mathcal{M}^{r} \simeq \mathcal{L}|_{\mathcal{Y}_r}$. Thus, $\mathcal{X}_r$ is isomorphic to an open subset of $\mathbb{V}^{\circ}(\mathcal{M}^{r})$ and is therefore a moduli stack of cyclic covers of degree $r$, which proves (1).
 
 For (4), surjectivity of $\mathcal{X}_r \to \mathcal{X}$ follows from the surjectivity of $\mathcal{Y}_r \to \mathcal{Y}$ which, in turn, follows by the $r$-divisibility of $\mathcal{X}$ (see Remark \ref{rem:rth_root}). To finish the proof it suffices to show that the $r$th power map on $\PicS_{\mathcal{U/\mathcal{Y}}}$ is proper quasi-finite and that it is \'etale  if $r\in \mathcal{O}_S^\times$. Indeed, the map $\mathcal{X}_r \to \mathcal{X}$ is the pullback of $(\cdot)^r$ (see Definition \ref{def:universalbranch}). For this, we may pass to a smooth cover of $\mathcal{Y}$. Thus, we may assume the map $\mathcal{U} \to \mathcal{Y}$ has a section and that $\mathcal{Y}$ is affine. In which case, as $\mathcal{U} \to \mathcal{Y}$ is representable by schemes, we have that $\mathcal{U}$ is also a scheme.  We then conclude by applying the following lemma.
 \end{proof}

\begin{lemma} \label{lem:picard} Fix a smooth proper morphism of schemes with geometrically connected fibres $\pi: X \to S=\Spec A$ which admits a section $\sigma$ and a relatively ample line bundle. Then the functor 
\[
	\PicS_{X/S} \to \Pic_{X/S} \times_S B\mathbb{G}_{m,S}, \quad L \mapsto ([L], \sigma^*L)
\]
is an equivalence of categories with inverse $([L],L') \mapsto L \otimes \pi^*\sigma^*L^{\vee} \otimes \pi^*L'$. Moreover, the morphism $(\cdot)^r: \PicS_{X/S} \to \PicS_{X/S}$
is proper quasi-finite and, if $r \in \mathcal{O}_S^{\times}$,   it is \'etale.
\end{lemma}

\begin{proof} We may assume that $S$ is Noetherian by choosing a Noetherian approximation $X_0 \to S_0$ of $X \to S$. That is, $S_0$ is of finite type over $\Spec \mathbb{Z}$, the morphism $X_0 \to S_0$ is smooth and proper with geometrically connected fibers, and $X = X_0 \times_{S_0} S$ over $S$; see, for example, \cite[Tags 01ZA and 01ZM]{stacks-project} for the existence of such an approximation and \cite[Prop. B.3]{rydh2015noetherian} for its properties. 

Since $\pi$ is cohomologically flat in dimension zero and admits a section, it follows that the Picard sheaf has a simple description 
\[\Pic_{X/S}(T \to S)=\Pic X_T/\Pic T.\]
Therefore the functors defined above are well-defined and one readily checks that they are inverse to each other.

Via this isomorphism, the $r$th power map on the Picard stack is the product of the $r$th power maps on $\Pic_{X/S}$ and $B\mathbb{G}_{m,S}$. We first check that the $r$th power map is proper. Since $(\cdot)^r: B\mathbb{G}_{m,S} \to B\mathbb{G}_{m,S}$ is a $\mu_r$-gerbe, hence proper, it suffices to check that the $r$th power map on the Picard sheaf is proper. Note that $\Pic_{X/S}$ is representable by a separated scheme locally of finite type over $S$ (see \cite[Thm.~9.4.8]{KlePic}) and satisfies the existence part of the valuative criterion by the argument in \cite[Tag 0DNG]{stacks-project}. As the map $(\cdot)^r: \Pic_{X/S} \to \Pic_{X/S}$ is of finite-type by \cite[Thm.~9.6.27]{KlePic}, we deduce that it is proper (since $\Pic_{X/S} \to S$ is separated and satisfies the existence part of the valuative criterion). The group homomorphism $(\cdot)^r$ is quasi-finite because $\text{Pic}^{\tau}_{X_s/k(s)}[r]$ is finite by the finiteness of $\text{Pic}^0_{X_s/k(s)}[r]$ and $\text{Pic}^{\tau}_{X_s/k(s)}/\text{Pic}^0_{X_s/k(s)}$ (see \cite[Cor.~9.6.17]{KlePic}). 

It remains to show that the $r$th power map on the Picard stack is formally \'etale when $r$ is invertible over $S$.
First note that the $r$th power map 
\begin{equation} \label{power} (\cdot)^r_X: B\mathbb{G}_{m,X} \to B\mathbb{G}_{m,X} \end{equation}
is a $\mu_r$-gerbe and in particular, since $r$ is invertible, it is formally \'etale (see \cite[Def.~B.5, Cor.~B.9]{rydh2011canonical}). Next, the Picard stack is by definition the Weil restriction $\PicS_{X/S}=\pi_*(B\mathbb{G}_{m,X})$ and the $r$th power map on the Picard stack is the Weil restriction of \eqref{power} i.e. $(\cdot)^r=\pi_*((\cdot)^r_X)$. It follows that $(\cdot)^r$ is formally \'etale because Weil restrictions preserve this property (see Lemma \ref{lem:weilrestriction} below).
 \end{proof}
 
That Weil restrictions preserve formally \'etale morphisms is well-known (see, for instance, \cite[Rem.~2.5]{hall2015general}) but we were unable to find a proof at the level of generality we require. As such, we include one below for the sake of completeness.

\begin{lemma} \label{lem:weilrestriction} Let $\pi: X \to S$ be a morphism of schemes and let $\phi: \mathcal{Z} \to \mathcal{Y}$ be a formally \'etale morphism of stacks over $X$. If $\phi$ is formally \'etale, then $\pi_*(\phi): \pi_*\mathcal{Z} \to \pi_*\mathcal{Y}$ is formally \'etale. \end{lemma}

\begin{proof}  
We verify that $\pi_*(\phi)$ is formally \'etale directly: let $T$ be a $\pi_* \mathcal{Y} $-scheme and $T_0 \to T$ be a closed immersion with nilpotent ideal sheaf. The induced map $X_{T_0} \to X_{T}$ is a closed immersion with nilpotent ideal sheaf and $X_T$ is a $\mathcal{Y}$-scheme, by definition of the Weil restriction. We have the following equivalences of categories
\begin{align*}
\text{Hom}_{\pi_*\mathcal{Y}}(T, \pi_*\mathcal{Z}) &= \text{Hom}_{\mathcal{Y}}(X_T, \mathcal{Z}) \\
& \xrightarrow{\sim} \text{Hom}_{\mathcal{Y}}(X_{T_0}, \mathcal{Z}) \\
& = \text{Hom}_{\pi_*\mathcal{Y}}(T_0, \pi_*\mathcal{Z}).
\end{align*}
The first and third equalities follow from the definition of the Weil restriction and the middle follows because $\phi$ is formally \'etale. Thus $\pi_* \phi$ is formally \'etale.
 \end{proof}

Under additional assumptions on the varieties being parametrised, we can say more about the morphism $\mathcal{X}_r \to \mathcal{X}$ defined in Definition \ref{def:universalbranch}.

\begin{corollary} \label{cor:gerbenotorsion}
	Let $\mathcal{X}$ be a moduli stack of divisorial pairs which is $r$-divisible and  parametrizes objects $(X \to S, H \subset X)$ with 
	$\Pic_{X/S}[r]=0$. 
	Then the morphism $\mathcal{X}_r \to \mathcal{X}$ is a $\mu_r$-gerbe.
\end{corollary}
\begin{proof}
	In this case, the proof of Lemma \ref{lem:picard} shows that 
	$(\cdot)^r: \PicS_{X/S} \to \PicS_{X/S}$ is a $\mu_r$-gerbe.
	Since $\mathcal{X}_r\to \mathcal{X}$ is   the pull-back of the $r$th power map $(\cdot)^r: \PicS_{X/S} \to \PicS_{X/S}$ (see Definition \ref{def:universalbranch}), this concludes the proof.    
\end{proof}

The hypothesis on the $r$-torsion of the Picard group in  Corollary \ref{cor:gerbenotorsion} holds, for example, when $X \to S$ is a Brauer-Severi scheme, or a (relative) $K3$ surface, or, more generally,  when $\mathrm{H}^1(X_{s},\mathcal{O}_{X_s})=0$ and $\text{NS}(X_{s})$ is torsion free.  
\subsection{Cyclic covers of projective spaces} 

In this section we consider the concrete case of cyclic covers of projective spaces. 

\begin{example} \label{ex:cyclicprojectivespace} 
Fix integers $n,r,d$ and take $\mathcal{X}=\mathcal{C}_{(rd;n-1), \mathbb{Z}[1/r]}$ (see Example \ref{ex:BS}). The universal stack $\mathcal{X}_r$  of cyclic covers (of degree $r$) in Theorem \ref{thm:properetale} is equivalent to $\mathcal{H}^{sm}(n,r,d)_{\mathbb{Z}[1/r]}$ from Example \ref{defn:ars}. Indeed, the functor which sends
\[(P \to S, L, s: \mathcal{O}_P \to L^{r}) \mapsto (P \to S, L^{r}, s: \mathcal{O}_P \to L^r)\]
identifies $\mathcal{H}^{sm}(n,r,d)_{\mathbb{Z}[1/r]}$ with $\mathcal{X}_r$ over $\mathcal{C}_{(rd;n-1), \mathbb{Z}[1/r]}$ because $P$ and $V(s)$ are $S$-smooth (see Proposition \ref{rem:smoothness}). Thus, Corollary \ref{cor:gerbenotorsion}  shows that associating the branch locus to a relative uniform cyclic cover of a Brauer-Severi scheme is a $\mu_r$-gerbe. 
\end{example} 

From Example \ref{ex:cyclicprojectivespace}, there is a natural morphism of stacks
$$\mathcal{H}^{\textrm{sm}}(n,r,d)_{\mathbb{Z}[1/r]} \to \mathcal{C}_{(rd;n-1), \mathbb{Z}[1/r]}$$
given by associating to a relative uniform cyclic cover its branch divisor. This gives $\mathcal{H}^{\textrm{sm}}(n,r,d)_{\mathbb{Z}[1/r]}$ the structure of a $\mu_r$-gerbe over $\mathcal{C}_{(rd;n-1), \mathbb{Z}[1/r]}$. If it is non-empty, we immediately see that the induced map
$$\pi_0(\mathcal{H}^{\textrm{sm}}(n,r,d)(k)) \to \pi_0(\mathcal{C}_{(rd;n-1)}(k))$$
is a bijection for any algebraically closed field $k$ with characteristic not dividing $r$. This recovers the well-known classical fact that over such an algebraically closed field, if a cyclic cover branched over a fixed divisor exists then it is unique up to isomorphism.

This is no longer true over non-algebraically closed fields in general, as the following shows. In particular, there is no section in general, which will greatly complicate our analysis and shows that there is no natural  way  to associate to an arbitrary family of hypersurfaces a family of cyclic covers ramified over exactly those hypersurfaces. 

\begin{proposition} 
	\hfil
	\begin{enumerate}
		\item Every non-empty fibre of
		$$\pi_0(\mathcal{H}^{\textrm{sm}}(n,r,1)(\Q)) \to \pi_0(\mathcal{C}_{(r;n-1)}(\Q))$$
		is infinite.
		\item If $\gcd(r,n+1) \neq 1$ and $\gcd(d,n+1) = 1$, then the map
		$$\pi_0(\mathcal{H}^{\textrm{sm}}(n,r,d)(\Q)) \to \pi_0(\mathcal{C}_{(rd;n-1)}(\Q))$$
		is not surjective. Thus $\mathcal{H}^{\textrm{sm}}(n,r,d) \to \mathcal{C}_{(rd;n-1)}$ admits no section over $\Q$.
		\item If $\gcd(r,n+1) = 1$, then 
		$\mathcal{H}^{\textrm{sm}}(n,r,1)_{\Z[1/r]} \to \mathcal{C}_{(r;n-1), \Z[1/r]}$
		admits a section.
	\end{enumerate}
 \end{proposition}
\begin{proof}
 	
	(1) As it is a $\mu_r$-gerbe, the fibre over any rational point is in bijection
	with $\H^1(\Q,\mu_r)$. By Kummer theory $\H^1(\Q,\mu_r) = \Q^\times/\Q^{\times r}$,
	which is infinite.
	
	(2) Assume that $\gcd(r,n+1)  = m > 1$ and $\gcd(d,n+1)=1$. Let $P$ be a Brauer--Severi
	variety of dimension $n$  
	and period $m$ over $\Q$. This can be constructed using the fundamental exact sequence from class field theory since the period of any class in $\Br \Q$ equals its index (see \cite[Thm.~1.5.36]{Poo17}).  
	The variety $P$ admits a smooth hypersurface $D \subset P$ 
	of degree $r$ since $m \mid r$  and by \cite[Thm.~5.4.10]{GS17} we have 
	$\Pic P=\mathbb{Z}\mathcal{O}(m)$.
	Suppose that $P$ admits a uniform
	cyclic cover $f: Z \to P$ of degree $r$ with branch locus $D$ of degree $rd$. 
	We have $f_*\mathcal{O}_Z=\mathcal{O}_P \oplus \mathcal{L} \oplus \dots \oplus \mathcal{L}^{r-1}$
	for some line bundle $\mathcal{L}$ of degree $-d$ on $P$ (see Example \ref{ex:branch_degree}). This implies that $m \mid d$ which contradicts our assumption that $\gcd(d,n+1)=1$. Therefore, $P$ does not admit a relative uniform cyclic cover $f:Z\to P$ of degree $r$ with branch locus $D$ of degree $rd$. In other words, the object $D \subset P$   of $C_{(rd;n-1)}(\Q)$ is not in the
	essential image of the functor $\mathcal{H}^{\textrm{sm}}(n,r,d)(\Q)  \to \mathcal{C}_{(rd;n-1)}(\Q)$.
	
	(3) Now suppose that $\gcd(r,n+1) = 1$. If $r=1$ the result is trivial, so assume $r > 1$.
	It suffices to prove that there is a relative uniform cyclic cover branched over exactly the universal hypersurface. So let 
	$\mathcal{D} \subset \mathcal{P} \to \mathcal{C}_{(r;n-1)}$ be the universal hypersurface of degree $r$ inside the relative Brauer-Severi scheme $\mathcal{P}$. Choose any global section  of $\mathcal{O}_{\mathcal{P}}(\mathcal{D})$ which defines $\mathcal{D}$ (e.g. dualizing the inclusion of the ideal sheaf $\mathcal{I}_{\mathcal{D}} \subset \mathcal{O}_{\mathcal{P}}$ yields one such example). By Remark \ref{rem:rth_root}, it remains to construct a line bundle $\mathcal{L}$ on $\mathcal{P}$ such that
	\begin{equation} \label{eqn:L}
		\mathcal{L}^{ r} \cong \mathcal{O}_{\mathcal{P}}(\mathcal{D}).
	\end{equation}	
	To do so, choose $s,t$  such that $sr - t(n+1) = 1$ and consider the line bundle
	$$\mathcal{J} = \OO_{\mathcal{P}}(\mathcal{D})^{s} \otimes \omega_{\mathcal{P}/\mathcal{C}_{(r;n-1)}}^{t}.$$
	This has degree $1$ along each geometric fibre. 
	Since the line bundle $\OO_{\mathcal{P}}(\mathcal{D}) \otimes \mathcal{J}^{ -r}$ has degree $0$ along each geometric fibre and $\mathcal{P} \to \mathcal{C}_{(r;n-1)}$ is a relative Brauer-Severi scheme, cohomology and base change implies that $\OO_{\mathcal{P}}(\mathcal{D}) \otimes \mathcal{J}^{-r}$ is the pullback of some line bundle on $\mathcal{C}_{(r;n-1)}$. However, by \cite[Thm.~5.1]{ArsieVistoli}, the Picard group of $\mathcal{C}_{(r;n-1), \Z[1/r]}$ is cyclic of order $(r-1)^n\gcd(r,n+1) = (r-1)^n$, which is coprime to $r$. In particular, all its elements are divisible by $r$. Therefore, the line bundle $\OO_{\mathcal{P}}(\mathcal{D}) \otimes \mathcal{J}^{-r}$  is divisible by $r$. This proves \eqref{eqn:L}, as required.
\end{proof}

\subsection{Application to arithmetic hyperbolicity}

We now state our most general results on arithmetic hyperbolicity for the moduli stacks considered in this section.

\begin{corollary} \label{cor:arithmetic_dp}
	Let $k$ be an algebraically closed field of characteristic $0$
	and $r \in \N$. Let $\mathcal{X}$ be a moduli stack of divisorial pairs over $k$
	which is $r$-divisible.
	Then $\mathcal{X}$ is arithmetically hyperbolic if and only if $\mathcal{X}_r$
	is arithmetically hyperbolic.
\end{corollary}
\begin{proof}
	By Theorem \ref{thm:properetale}, the morphism $\mathcal{X}_r\to \mathcal{X}$ is proper \'etale surjective, so that the corollary follows directly from  Chevalley--Weil (Theorem \ref{thm:chev_weil}).
\end{proof}

As for integral points, we have the following.

\begin{corollary} \label{cor:finite_dp} 
	Let $r_1,r_2 \in \N$ and let   $A$ be an integrally closed finitely generated $\Z$-algebra. Let   $k = \overline{\kappa(A)}$, which we assume has characteristic $0$.
	Let $\mathcal{X}$ be a moduli stack of divisorial pairs over $A$ 
	which is $r_1$-divisible. 
	If $\mathcal{X}_{r_1}$ has finite diagonal and is arithmetically hyperbolic over $k$,
	then $\pi_0(\mathcal{X}_{r_2}(A))$ is finite.
\end{corollary}
\begin{proof}
	
	The stack $\mathcal{X}$ has finite diagonal and is arithmetically
	hyperbolic over $k$ since $\mathcal{X}_{r_1}$ has these properties. Indeed, $\mathcal{X}_{r_1} \to \mathcal{X}$ is proper, quasi-finite and surjective (see Theorem \ref{thm:properetale}) so these properties descend (see Corollary \ref{cor:arithmetic_dp} and Lemma \ref{lem:finitediagonal}). Moreover, as 
	$\mathcal{X}_{r_2} \to \mathcal{X}$ is proper and quasi-finite 
	by Theorem \ref{thm:properetale}, 
	we see that $\mathcal{X}_{r_2}$ has finite diagonal (by Lemma \ref{lem:finitediagonal}) and is arithmetically hyperbolic (by Lemma \ref{lem:qf}). The result now follows from the
	twisting lemma (Theorem \ref{thm:twists}).
\end{proof}

The following is  well-known; we include a proof for completeness.

\begin{lemma} \label{lem:finitediagonal} Let $f: X \to Y$ be a proper quasi-finite morphism of stacks over a scheme $S$.  If $Y$ has finite diagonal,  then so does $X$. The converse holds if $f$ is surjective. 
\end{lemma}

\begin{proof}
Recall that  an algebraic stack is said to be quasi-DM if its diagonal is locally quasi-finite (see \cite[Tag 01TD]{stacks-project}). An algebraic stack $Z$ has finite diagonal if and only if $Z$ is quasi-DM and separated (see \cite[Tag 02LS]{stacks-project}).

Suppose $Y$ has finite diagonal. Since the morphism $X \to Y$ and the stack $Y$ are quasi-DM and separated, the same is true for $X$ (see \cite[Tag 050L]{stacks-project}). Thus, the diagonal of $X$ is finite. 

Now suppose that $X$ has finite diagonal. We first prove that $Y$ is quasi-DM (i.e.~that $\Delta_{Y/S}$ is locally quasi-finite). Consider the following diagram
\[  \xymatrix{ X \ar[d]_f \ar[rr]^{\Delta_{X/S}} &  &  X \times_S X \ar[d]^{f \times f} \\  Y \ar[rr]^{\Delta_{Y/S}} &  &  Y \times_S Y. }  \]
As $X$ has finite diagonal and $X \to Y$ is proper quasi-finite, it follows that the geometric fibers of $X \to Y \times_S Y$ are finite. As $X \to Y$ is proper, quasi-finite and surjective, it follows that the geometric fibres of $\Delta_{Y/S}$ are finite and discrete, and since $\Delta_{Y/S}$ is representable it follows that $\Delta_{Y/S}$ is quasi-finite, i.e.~$Y$ is quasi-DM. Moreover $\Delta_{X/S}$ and $f \times f$ are universally closed so the same is true for $\Delta_{Y/S}$.

Similarly, because $\Delta_{\Delta_X}$ and $f \times f$ are universally closed, the same is true for $\Delta_{\Delta_{Y/S}}$. Hence \cite[Tag 04Z0 (1)]{stacks-project} implies $Y$ is separated, as required.
\end{proof}

\section{Cyclic covers of projective space}
\label{section:pres}

In this section we study stacks of smooth hypersurfaces in projective space and their associated cyclic covers. We use these properties to then prove Theorem \ref{thm:Shaf_via_Fano}.

\subsection{Definition and basic properties}

For an integer $r\geq 2$ and $n\geq 1$, we let $\mathrm{Hilb}_{r,n} $ be the Hilbert scheme of hypersurfaces of degree $r$ in $\mathbb{P}^{n+1}$ over $\mathbb{Z}$. Let $\mathrm{Hilb}^{\textrm{sm}}_{r,n}$ be the open subscheme parametrizing smooth hypersurfaces.  
The  automorphism   group scheme $\PGL_{n+2} $ of $\mathbb{P}^{n+1}_{\ZZ}$   acts on $\mathrm{Hilb}_{r,n}^{\textrm{sm}}$, and we let $\mathcal{C}'_{(r;n)}$ denote the quotient stack $[\mathrm{Hilb}^{\textrm{sm}}_{r,n}/ \PGL_{n+2}  ]$. We refer to $\mathcal{C}'_{(r;n)}$ as the stack of smooth hypersurfaces of degree $r$ in $\mathbb{P}^{n+1}$ over $\mathbb{Z}$.

By construction, the stack $\mathcal{C}'_{(r;n)}$ is a smooth finite type  algebraic stack over $\mathbb Z$ with affine diagonal. If $r \geq 3$, then $\mathcal{C}'_{(r;n)}$ is separated by \cite[\S 16]{GIT} (see also \cite[Thm.~1.7]{Ben13}). If $(r,n) \neq (3,1)$, then $\mathcal{C}'_{(r;n)}$ is Deligne--Mumford over $\ZZ$, whereas $\mathcal{C}'_{(3;1)}$ is Deligne--Mumford over $\Z[1/3]$ \cite[Thm.~1.6]{Ben13}.
First, we describe the functor of points of $\mathcal{C}'_{(r;n)}$.

	\begin{lemma} \label{descriptionofhyper}
		The stack $\mathcal{C}'_{(r;n)}$  is isomorphic to $\mathcal{C}_{(r;n)}$ (see Example \ref{ex:BS}).
	\end{lemma}
	
	\begin{proof} Consider the morphism $\phi: \mathcal{C}_{(r;n)} \to B\PGL_{n+2}$ which sends a pair $(P \to S, H \to P)$ to the $\PGL_{n+2}$-torsor $\underline{\mathrm{Isom}}_{S}(\mathbb{P}^{n+1}_S, P)$ over $S$ corresponding to the Brauer--Severi scheme $P \to S$. 
		Now, consider the $2$-fibre product \[\mathcal{C}_{(r;n)} \times_{B\PGL_{n+2}} \Spec \mathbb{Z},\] where the morphism $\Spec \mathbb{Z}  \to B\PGL_{n+2}$ is the universal $\PGL_{n+2}$-torsor, i.e.~it sends a scheme $S$ to the trivial $\PGL_{n+2}$-torsor. If $S$ is a scheme, then the $S$-objects of this  fibre product are tuples 
		\[(P \to S, H \to P, f : P \to \mathbb{P}^{n+1}_{S}),\]
		 where $P \to S$ is a Brauer--Severi scheme of relative dimension $n$, the scheme $H$ is a smooth hypersurface of degree $r$, and $f$ is an isomorphism of schemes over $S$. A morphism between two triples 
		\[(P \to S, H \to P, f) \to (P' \to S, H' \to P', f')\]
		is the data of two  isomorphisms $a: H \to H'$ and $b: P \to P'$ 
		such that the diagram
\[\xymatrix{ H \ar[d]_a \ar[r] & P \ar[d]_b \ar[r]^f & \P^{n+1}_S \ar[d]^{\mathrm{id}} \\
H' \ar[r] & P' \ar[r]^{f'} & \P^{n+1}_S  }
		\]  
		 commutes  in the category of $S$-schemes. From this diagram, we see that  objects in this groupoid have no nontrivial automorphisms. Moreover, every object in this groupoid is uniquely isomorphic to a triple of the form $(\mathbb{P}^n_S \to S, H \to \mathbb{P}^n_S, \mathrm{id}_{\mathbb{P}^n_S})$. Thus, we may  identify the stack with its associated sheaf. Now,  the morphism of functors
		\[\mathcal{C}_{(r;n)} \times_{B\PGL_{n+2}} \Spec \mathbb{Z} \to \Hilb^{\textrm{sm}}_{(r;n)}\]
	defined by
		\[(\mathbb{P}^{n+1}_{S} \to S, H_{S} \to \mathbb{P}^n_{S}, \mathrm{id}: \mathbb{P}^{n+1}_{S} \to \mathbb{P}^{n+1}_{S}) \mapsto   
		(H_{S} \to \mathbb{P}^{n+1}_{S}) \]
		is an isomorphism compatible with the respective actions of $\PGL_{n+2}$.
		Thus
		\[\mathcal{C}_{(r;n)}'=[\Hilb^{\textrm{sm}}_{(r;n)}/\PGL_{n+2}] \cong \mathcal{C}_{(r;n)}. \] This concludes the proof. 
	\end{proof}

Henceforth, we identify $\mathcal{C}'_{(r;n)}$ and $\mathcal{C}_{(r;n)}$.

\subsection{From cyclic covers back to hypersurfaces}

Let $\mathbb{A}^{\circ}(r,n)$ denote $\mathbb{V}^{\circ}(\pi_*\mathcal{O}_{\mathbb{P}^n_{\mathbb{Z}}}(r))$ where $\pi: \mathbb{P}^n_{\mathbb{Z}} \to \Spec \mathbb{Z}$. We view this as the (affine) space of nonzero degree $r$ forms in $n+1$ variables, and let $\mathbb{A}^{\text{sm}}(r,n)$ denote the open subset consisting of forms which define smooth hypersurfaces in $\mathbb{P}^n_{\mathbb{Z}}$.

Given a hypersuface $f(x) = 0$ of degree $r$, one can view the cyclic cover $f(x) = x_{n+1}^r$ as a hypersurface in projective space in its own right. We now show that this construction gives rise to morphism of stacks. 
Consider the map $\phi: \mathbb{A}^{\textrm{sm}}(r,n)_{\mathbb{Z}[1/r]} \to \mathbb{A}^{\textrm{sm}}(r,n+1)_{\mathbb{Z}[1/r]}$ which sends a smooth degree $r$ form $f$ in $x_0, \dots ,x_n$
		to $x_{n+1}^r-f$. Moreover, consider the  group homomorphism $\eta: \GL_{n+1} \to \GL_{n+2}$  defined by 
		\[
		A \mapsto \left(\begin{array}{cc} A & 0   \\ 0 & 1  \end{array} \right).
		\]
	 Observe that the morphism $\phi: \mathbb{A}^{\textrm{sm}}(r,n)_{\mathbb{Z}[1/r]} \to \mathbb{A}^{\textrm{sm}}(r,n+1)_{\mathbb{Z}[1/r]}$ \ is equivariant with respect to $\eta$. In particular, $\phi$ induces a  morphism	
	\begin{align*}
	&\mathcal{H}^{\textrm{sm}}(n,r,1)_{\mathbb{Z}[1/r]} \\
	&=[\mathbb{A}^{\textrm{sm}}(r,n)_{\mathbb{Z}[1/r]}/\GL_{n+1}] \longrightarrow [\mathbb{A}^{\textrm{sm}}(r,n+1)_{\mathbb{Z}[1/r]}/\GL_{n+2}]=\mathcal{H}^{\textrm{sm}}(n+1,r,1)_{\mathbb{Z}[1/r]}.
	\end{align*}
(See \cite[Thm.~4.1]{ArsieVistoli} for this explicit presentation as a quotient stack.) 
		On the level of geometric points, this maps a cyclic cover $Z \to \mathbb{P}^n_k$ branched along the divisor defined by the degree $r$ form $f(x)$ to the cyclic cover $\tilde{Z} \to \mathbb{P}^{n+1}_k$ branched along the divisor defined by the form $x_{n+1}^r-f$. 
		
		Composing  with the $\mu_r$-gerbe $\mathcal{H}^{\textrm{sm}}(n+1,r,1)_{\mathbb{Z}[1/r]}\to \mathcal{C}_{(r;n),\mathbb{Z}[1/r]}$ from Example \ref{ex:cyclicprojectivespace} gives  a morphism
		\[
		\mathcal{H}^{\textrm{sm}}(n,r,1)_{\mathbb{Z}[1/r]} \to \mathcal{C}_{(r;n),\mathbb{Z}[1/r]}.
		\]
On the level of geometric points, this morphism sends a cyclic cover $Z \to \P^n$ branched along the degree $r$ form $f(x_0,...,x_n)$ to the hypersurface in $\P^{n+1}$ defined by $x_{n+1}^r-f$.

	\begin{proposition} \label{unramifiedmapofstacks}
		The morphism $\mathcal{H}^{\textrm{sm}}(n,r,1)_{\mathbb{Z}[1/r]}  \to \mathcal{C}_{(r;n), \mathbb{Z}[1/r]}$ is unramified.
	\end{proposition}
	\begin{proof}   
Since $\mathcal{H}^{\textrm{sm}}(n,r,1)_{\mathbb{Z}[1/r]}$ and $\mathcal{C}_{(r;n), \mathbb{Z}[1/r]}$ are both Deligne--Mumford stacks, there is a commutative square
 \[  \xymatrix{ U \ar[d]_ a \ar[r]^g &  V \ar[d]^ b \\  \mathcal{H}^{\textrm{sm}}(n,r,1)_{\mathbb{Z}[1/r]} \ar[r]^f &  \mathcal{C}_{(r;n), \mathbb{Z}[1/r]} }  \]
where $a$ and $b$ are \'etale covers by schemes. By \cite[Tag 0CIT]{stacks-project} we need to show that $g$ is unramified. However, because $a$ and $b$ each induce an isomorphism on tangent spaces, it is enough to show that, for every geometric point \[x\colon~\Spec k \to \mathcal{H}^{\textrm{sm}}(n,r,1)_{\mathbb{Z}[1/r]},\] the induced map on tangent spaces $\text{d}f$ is injective. Indeed, in that case $\text{d}g$ is injective as well and therefore $g$ is unramified by \cite[Tag~0B2G]{stacks-project}, as desired.

Note that $x$ corresponds to a cyclic cover $Z \to \mathbb{P}^n_k$. We let  $V(f(x_0,...,x_n))=H \subset \mathbb{P}^n_k$ denote its branch locus. 
Since $\mathcal{H}^{\textrm{sm}}(n,r,1)_{\mathbb{Z}[1/r]}  \to \mathcal{C}_{(r;n-1), \mathbb{Z}[1/r]}$ is \'etale by Theorem \ref{thm:properetale}, we have  a natural isomorphism of tangent spaces 
		\[T_{Z \to \mathbb{P}^n_k}(\mathcal{H}^{\textrm{sm}}(n,r,1)) \to T_{H \subset \mathbb{P}^n_k}(\mathcal{C}_{(r;n-1)})\]
		   By \cite[Prop.~3.4.17]{sernesidef}, we have  \[T_{H \subset \mathbb{P}^n_k}(\mathcal{C}_{(r;n-1)}) = \mathrm{H}^1(\mathbb{P}^n_k, T_{\mathbb{P}^n_{k}}\langle H\rangle ),\] where 
		\[T_{\mathbb{P}^n_{k}}\langle H\rangle=\ker[T_{\mathbb{P}^n_k} \to N_{H/\mathbb{P}^n_k}].\]

  Let $\widetilde{Z}\to \mathbb{P}^{n+1}_k$ be a degree $r$ uniform cyclic cover of $\mathbb{P}^{n+1}_k$ ramified precisely along $\tilde{H}=V(x_{n+1}^r-f) \subset \P^{n+1}_k$. We have the following   diagram of linear maps
 \[
 \xymatrix{   T_{Z \to \mathbb{P}^n_k}(\mathcal{H}^{\textrm{sm}}(n,r,1))  \ar[d]_{\cong} \ar[rrr]^{\textrm{the differential     }} & &  &   T_{\widetilde{Z} \to \mathbb{P}^{n+1}_k}(\mathcal{H}^{\textrm{sm}}(n+1,r,1))   \ar[d]^{\cong} \\  T_{H \subset \mathbb{P}^n_k}(\mathcal{C}_{(r;n-1)}) & & & T_{\tilde{H} \subset \mathbb{P}_k^{n+1}}(\mathcal{C}_{(r;n)})    }
 \]
 It follows from this diagram that, to prove the proposition, it suffices to show that the following map
		\[   T_{H \subset \mathbb{P}^n}(\mathcal{C}_{(r;n-1)})=\mathrm{H}^1(\mathbb{P}^n_k, T_{\mathbb{P}^n_k}\langle H\rangle) \to T_{\tilde{H} \subset \mathbb{P}^{n+1}}(\mathcal{C}_{(r;n)}) = \mathrm{H}^1(\mathbb{P}^{n+1}_k, T_{\mathbb{P}^{n+1}_k}\langle \tilde{H}\rangle)\] is injective.
		The latter group maps to $\mathrm{H}^1(\tilde{H},T_{\tilde{H}})$, being the deformation space of $\tilde{H}$ viewed as an abstract variety. Moreover, $Z$ and $\tilde{H}$ are abstractly isomorphic: restricting the projection $\mathbb{P}^{n+1}_k \dasharrow \mathbb{P}^n_k$ to $\tilde{H}$ gives $\tilde{H}$ the structure of a cyclic cover of degree $r$ branched along $H$. Thus to prove the proposition, it suffices to show that the natural map \[\mathrm{H}^1(\mathbb{P}^n_k, T_{\mathbb{P}^n_k}\langle H\rangle) \to \mathrm{H}^1(Z, T_Z)\]
		is injective. 
		 However, this follows from the fact that  the sheaf $T_{\mathbb{P}^n_k}\langle H\rangle $ is a direct summand of $f_*T_Z$ (see \cite[Lem.~3.1]{Ivinskis} or \cite[Prop.~3.8]{Wehler}).
	\end{proof}

	\begin{remark}
	Note that Proposition \ref{unramifiedmapofstacks} improves \cite[Lem.~2.3]{ViehwegZuoCM1}. 
	\end{remark}

\subsection{Arithmetic hyperbolicity of hypersurfaces and good reduction}

We now begin our arithmetic applications. Our first result makes clear the precise relationship between arithmetic hyperbolicity of the stack of smooth hypersurfaces, the Shafarevich conjecture (Conjecture \ref{conji}), and various possible finiteness statements concerning good reduction. The key issues are performing a descent from the algebraic closure, glueing different models, and dealing with integral points corresponding to hypersurfaces which are in a Brauer--Severi scheme and not a projective space. 

\begin{definition}
	Let $A$ be an integral domain 
	with fraction field $K$ and $X \subset \P^{n+1}_K$ a smooth hypersurface.
	We say that $X$ has \emph{good reduction over} $A$ if there exists a smooth hypersurface
	$\mathcal{X} \subset \P^{n+1}_A$ whose generic fibre is $K$-linearly isomorphic to $X$.
	We call such an $\mathcal{X}$, together with the choice of isomorphism,
	 a \emph{good model} for $X$.
\end{definition}
 
In the statement, for brevity by a pair $(K,S)$ we mean a number field $K$ and finite set of places $S$ of $K$ containing all the archimedean places.

\begin{proposition} \label{prop:Rewriting}
	Let $r \geq 3$ and $n \geq 1$.  Then the following are equivalent.
	\begin{enumerate}
		\item  For every algebraically closed field $k$ of characteristic zero, the stack $\mathcal{C}_{(r;n), k}$ is arithmetically hyperbolic
		over $k$.
		\item For every integrally closed finitely generated $\Z$-algebra $A$ with $\kappa(A)$
		of characteristic $0$, the set
		$\pi_0(\mathcal{C}_{(r;n)}(A))$ is finite. 
		\item  For all pairs $(K,S)$,
		the set of $\OO_{K,S}$-linear isomorphism classes of smooth
		hypersurfaces of degree $r$ in $\mathbb{P}^{n+1}_{\OO_{K,S}}$ is
		finite. 
		\item For all pairs $(K,S)$,
		the set of $K$-linear isomorphism classes of smooth
		hypersurfaces of degree $r$ in $\mathbb{P}^{n+1}_{K}$ with good
		reduction over $\O_{K,S}$ is finite.
		\item For all pairs $(K,S)$,
		the set of $K$-linear isomorphism classes of smooth
		hypersurfaces of degree $r$ in $\mathbb{P}^{n+1}_{K}$ with good
		reduction over the localisations $\O_{K,v}$ for all $v \notin S$ is finite.
	\end{enumerate}
\end{proposition}
\begin{proof}
$(1)\implies (2)$: This follows from the Twisting Lemma (Theorem \ref{thm:twists}), since $\mathcal{C}_{(r;n)}$ is separated over $\ZZ$ \cite[Thm.~1.7]{Ben13} with affine diagonal (by construction).   

$(2) \implies (3)$: Take $A = \O_{K,S}$.

$(3) \implies (4)$: This is immediate, as $(3)$ implies that there are only finitely many good models up to linear isomorphism.

$(4) \implies (5)$: Let $K$ be a number field with finite set of places $S$. To prove $(5)$, we are free to increase $S$. In particular, we may assume that $\Pic \O_{K,S} = 0$.

Let $X$ be a smooth hypersurface over $K$ with good models $\mathcal{X}_v$ over $\O_{K,v}$ for all $v \notin S$. By spreading out, there is a finite set of places $S \subset T$ and good model for $X$ over $\O_{K,T}$. Glueing this model with the models at $v \in T \setminus S$, we obtain a smooth proper scheme $f: \mathcal{X} \to \text{Spec } \O_{K,S}$ whose fibres are smooth hypersurfaces. Let $\O_{\mathcal{X}}(1)$ denote the line bundle determined by the closure in $\mathcal{X}$ of a hyperplane in $X$. Then by \cite[Lem.~1.1.8]{BenoistThesis}, this is relatively very ample, hence induces an embedding $\mathcal{X} \to \P(f_*(\O_{\mathcal{X}}(1)))$ into a projective bundle that is Zariski locally a hypersurface in projective space. But as $\Pic \O_{K,S} = 0$ and $\O_{K,S}$ is Dedekind, the locally free module  $f_*(\O_{\mathcal{X}}(1))$ is free. Thus $\mathcal{X}$ is in fact a good model for $X$, hence $(4) \implies (5)$.

$(5) \implies (4)$: Immediate, as a good model over $\O_{K,S}$ gives rise to a good model over all $\O_{K,v}$ for $v \notin S$.

	$(4) \implies (1)$: By \cite[Thm.~7.2]{JLitt}, it suffices to show that the stack $\mathcal{C}_{(r;n),\Qbar}$ is arithmetically hyperbolic over $\Qbar$. To do so, 
	let $K$ be a number field and $S$ a finite set of places of $K$.
	By Lemma \ref{descriptionofhyper}, an $\OO_{K,S}$-point
	of $\mathcal{C}_{(r;n)}$ 
	is given by a Brauer--Severi scheme $P$ over $\Spec \OO_{K,S}$
	of relative dimension $n+1$ together with a relative hypersurface
	$H \subset P$ of degree $r$. But such Brauer--Severi schemes are parametrised
	by $\H^1(\O_{K,S}, \PGL_{n+1})$, which is finite by Lemma \ref{lem:GMB}.
	Thus there exists a finite field extension 
	$K'/K$ which trivialises each such $P$. Let $S'$ be a finite set of places
	containing all places above $S$.
	Consider the diagram
	\[
	\xymatrix{ \mathcal{C}_{(r;n)}(\OO_{K,S}) \ar[d] \ar[r]  & \mathcal{C}_{(r;n)}(\OO_{K',S'}) \ar[d]  \\ 
	\mathcal{C}_{(r;n)}(K) \ar[r] &\mathcal{C}_{(r;n)}(K').  }
	\]  
	We have shown that elements in the image of the horizontal arrows
	may be represented by hypersurfaces in $\P^{n+1}_{\O_{K',S'}}$ and $\P^{n+1}_{K'}$, respectively.
	Applying $(4)$ to  $(K',S')$ now implies that
	$$\Im[ \pi_0(\mathcal{C}_{(r;n)}(\OO_{K,S}))~\to~\pi_0(\mathcal{C}_{(r;n)}(K'))]$$
	is finite, which proves (1).
\end{proof}

\begin{remark}
	Using Proposition \ref{prop:Rewriting}, we may slightly improve on some results in the literature. For example, in \cite[Cor.~1.3.2]{Andre} Andr\'e proves that there are only
	finitely many linear isomorphism classes of smooth quartic surfaces over $\Z[1/2m]$
	for any $m \in \Z$. Combining Proposition \ref{prop:Rewriting} with \cite[Thm.~1.3.1]{Andre}
	shows, for example, that there are only finitely many linear isomorphism classes of smooth 
	quartic surfaces over $\Z[1/m]$ for any $m \in \Z$, i.e.~inverting $2$ is not necessary.
\end{remark}

For completeness, we also verify Conjecture \ref{conji} for plane curves.

\begin{proposition}\label{prop:plane_curves} 
	Conjecture \ref{conji} holds for $n = 1$ and any $r \geq 2$.
\end{proposition}
\begin{proof}
	The case $r=2$ (plane conics) is a special case of \cite[Prop.~5.1]{JL}.
	For $r=3$ (plane cubics), consider the forgetful map $\mathcal{C}_{(3;1),\bar{\Q}}\to\mathcal{M}_{1,1,\bar{\Q}}$, where $\mathcal{M}_{1,1}$ is the stack of elliptic curves over $\mathbb{Z}$. By  \cite[Props.~6.1, 6.4]{Bergh},  this morphism is a gerbe under the universal $3$-torsion group scheme, and thus quasi-finite. As $\mathcal{M}_{1,1,\bar{\Q}}$ is arithmetically hyperbolic by Shafarevich \cite[Thm.~IX.6.1]{Silverman}, it follows from Lemma \ref{lem:qf} that  $\mathcal{C}_{(3;1),\bar{\Q}}$ is arithmetically hyperbolic. The result now follows from Proposition \ref{prop:Rewriting}.

		For $r\geq 4$. Let $g:=(d-1)(d-2)/2$ and let $\mathcal{M}_g$ be the stack of smooth proper connected curves of genus $g$ over $\ZZ$. By \cite{Chang}, two smooth plane curves in $\mathbb{P}^2_{\bar{\Q}}$ are birational if and only if they are linearly isomorphic. In particular,   the forgetful morphism $\mathcal{C}_{(d;1),\bar{\Q}}\to\mathcal{M}_{g,\bar{\Q}}$ is injective on isomorphism classes of $\bar{\Q}$-points and fully faithful. This implies that $\mathcal{C}_{(d;1),\bar{\Q}}\to \mathcal{M}_{g,\bar{\Q}}$ is  quasi-finite, so that the result again follows from  Faltings and Proposition \ref{prop:Rewriting}.  
	\end{proof}
	
Our method gives the following general result.

\begin{theorem} \label{thm:meta}
	Let $n \geq 1, r \geq 2$ and $d \geq 1$ with $rd \neq 2$. Let $k$ be an algebraically
	closed field of characteristic $0$ and  $A \subset k$  a $\Z$-finitely generated integrally closed subring.
	If	$\mathcal{C}_{(rd;n)}$ is arithmetically hyperbolic over $k$ 
	then $\pi_0(\mathcal{H}^{sm}(n+1,r,d)(A))$ is finite.
\end{theorem}
\begin{proof}	
By Example \ref{ex:cyclicprojectivespace}, the stack $\mathcal{H}^{sm}(n+1,r,d)_k$ is isomorphic to the stack $\mathcal{X}_r$ from Theorem \ref{thm:properetale}. Thus the morphism $\mathcal{H}^{sm}(n+1,r,d)_k \to \mathcal{C}_{(rd;n),k}$ is proper \'etale, hence $\mathcal{H}^{sm}(n+1,r,d)_k$ is arithmetically hyperbolic by Lemma \ref{lem:qf}. As $\mathcal{H}^{sm}(n+1,r,d)$ has finite diagonal (Example \ref{defn:ars}), the result follows from Theorem \ref{thm:twists}. 
\end{proof}

\subsection{Proof of Theorem \ref{thm:Shaf_via_Fano}}

First note that Conjecture \ref{conji} is known for quadrics \cite[Prop.~5.1]{JL}, so we may assume $r \geq 3$. We will prove the result using the stacky Chevalley--Weil theorem and the cyclic covering trick.

	\begin{proposition}\label{prop:shaf_conjec} Let $k$ be an algebraically closed field of characteristic zero. 
		Let $r\geq 3$ and $N\geq 2$. Suppose that $\mathcal{C}_{(r;N),k}$ is arithmetically hyperbolic over $k$. Then, for all $1\leq n \leq N$, the stack $\mathcal{C}_{(r;n),k}$ is arithmetically hyperbolic over $k$.
	\end{proposition}
	\begin{proof}  By induction,   it suffices to show that $\mathcal{C}_{(r;N-1),k}$ is arithmetically hyperbolic over $k$.
		To do so,  consider the $\mu_r$-gerbe  from Example \ref{ex:cyclicprojectivespace} and the unramified   (hence quasi-finite) map from
		Proposition \ref{unramifiedmapofstacks}.
 In a diagram: 
		\[
		\xymatrix{ \mathcal{H}^{\textrm{sm}}(N,r,1)_{k} \ar[d]_{\textrm{proper \'etale}} \ar[rr]^{\,\,\,\,\,\,\,\,\textrm{quasi-finite}}   & & \mathcal{C}_{(r;N),k} \\  \mathcal{C}_{(r;N-1),k}.&  &  }
		\]  
		Since $\mathcal C_{(r;N),k}$ is arithmetically hyperbolic over $k$, it follows that the stack $\mathcal{H}^{\textrm{sm}}(N,r,1)_k$  is arithmetically hyperbolic (Lemma \ref{lem:qf}). Then the stacky Chevalley--Weil (Theorem \ref{thm:chev_weil}) shows that $\mathcal C_{(r;N-1),k}$ is arithmetically hyperbolic  over $k$.
	\end{proof}

We now prove Theorem \ref{thm:Shaf_via_Fano}. The stack $\mathcal{C}_{(r;N),{\bar{\Q}}}$ is arithmetically hyperbolic by Proposition \ref{prop:Rewriting} and our assumptions.   Proposition \ref{prop:shaf_conjec} implies that $\mathcal{C}_{(r;n),\bar{\Q}}$ is arithmetically hyperbolic for all $1\leq n \leq N$, so that Theorem \ref{thm:Shaf_via_Fano}  follows readily from Proposition \ref{prop:Rewriting}.	
\qed

\section{Covers of the plane}

We specialise Theorem \ref{thm:meta} to covers of the plane, to deduce new (and old) finiteness statements about explicit families of surfaces.
Although each of the stacks $\mathcal{H}(2,r,d)$ parametrize surfaces, they also admit a proper \'etale map to moduli stacks parametrizing curves. Thus, we are able to deduce finiteness results from  curves.
 
\begin{theorem} \label{thm:H2rd} 
	Let $r \geq 2$ and $d \geq 1$ with $rd \neq 2$. Let $k$ be an algebraically
	closed field of characteristic $0$ and  $A \subset k$  a $\Z$-finitely generated integrally closed subring. Then $\pi_0(\mathcal{H}^{sm}(2,r,d)(A))$ is finite.
\end{theorem} 
\begin{proof}
	Here $\mathcal{C}_{(rd;1)}$ is arithmetically hyperbolic over $k$ by Propositions \ref{prop:Rewriting} and \ref{prop:plane_curves}. The result therefore follows from
	Theorem \ref{thm:meta}.
\end{proof}

\subsection{Double covers}

This result is especially interesting when $r=2$ because all double covers are automatically $\mu_2$-cyclic covers.

\begin{lemma} \label{lem:doublecovers} Let $S$ be a scheme with $2 \in \mathcal{O}_S^{\times}$. Let $f: X \to Y$ be a morphism of smooth finitely presented schemes over $S$ which is a degree $2$ finite flat morphism. Then $X \to Y$ has a unique structure of a relative uniform cyclic cover of $Y$, up to $S$-isomorphism.
\end{lemma}

\begin{proof}  
To prove the lemma,    we first show that $X$ admits at most one non-trivial action of $\mu_2$ making $f$ $\mu_2$-invariant. 
When $f$ is \'etale this follows because any unramified double covering is Galois with $\underline{\mathrm{Aut}}_{X/Y} =\mu_2$, and therefore only one action is possible. If $f$ is not \'etale, as $2$ is invertible and $X$ and $Y$ are $S$-smooth, there is an open subset $U \subset Y$ (dense in every fibre $Y_s$) over which $f$ is \'etale, so any automorphism $X \to X$ over $Y$ is unique and an involution. This immediately implies that if $X \to Y$ admits the structure of a cyclic cover, then it is unique.

To show that $X \to Y$ admits such a structure, we show that it admits a $\mu_2$-action as in Definition \ref{def:cyclic_cover}. By uniqueness, it suffices to find the involution locally, so we may pass to an affine covering of $Y$ and so we set $Y=\Spec A$. Next we show that $X \simeq \Spec A[x]/(x^2-a)$ for some non-zero divisor $a \in A$. This will prove the existence of the desired $\mu_2$-action on $X$ since the right hand side has the obvious involution $x \mapsto -x$ over $A$.

To obtain such an isomorphism, observe that half the trace gives a left splitting of  the short exact sequence 
\[0 \to \mathcal{O}_Y \to f_*\mathcal{O}_X \to L \to 0,\]
where $L$ is a line bundle. By passing to a further covering of $\Spec A$ we may assume $L$ has a nowhere vanishing section $b'$, so that $f_*\mathcal{O}_X=\mathcal{O}_Y\oplus\mathcal{O}_Yb'$ (as modules). To determine the $\mathcal{O}_Y$-algebra structure of $f_*\mathcal{O}_X$ note that, for some $a', a'' \in A$, we have
\[b'^2+a''b'-a'=0\]
Set $b=b'-\frac{a''}{2}$ so $b^2-a=0$ for some $a \in A$, $f_*\mathcal{O}_X=\mathcal{O}_Y \oplus \mathcal{O}_Yb$,  and therefore $f_*\mathcal{O}_X \simeq A[x]/(x^2-a)$,  as desired.  
Let $Z = V(a) \subset Y$. To see that this gives $f$ the structure of a relative uniform cyclic cover, we need to check that $Z$ is a relative Cartier divisor. However, $Z$ is the branch locus of $f$ and is therefore Cartier on every fibre $Y_s$ by purity of the branch locus and the $S$-smoothness of $X$ and $Y$. Then $Z$ is Cartier and $S$-flat by \cite[Tag 062Y]{stacks-project}.  
\end{proof}

\begin{remark} \label{rem:doublecovers} 
It follows that any finite flat map of degree two (of finite presentation) $f: X \to Y$ between $S$-smooth schemes has the unique structure of a relative uniform cyclic cover. In particular, this implies that if $f':X' \to Y$ is another such morphism and $X \to X'$ is map over $Y$ then it is automatically a map of relative uniform cyclic covers. We will use this fact repeatedly to show that certain moduli stacks of relative uniform cyclic covers of degree $2$ have alternative descriptions.  
\end{remark}
\subsection{Special cases}

We now explain how to recover versions of the Shafarevich conjecture for del Pezzo surfaces and K3 surfaces of degree $2$.

\begin{definition}  Let $\mathcal{D}_2$ denote the category fibred in groupoids of \emph{Del Pezzo surfaces of degree $2$}. That is, the groupoid over a scheme $T$ consists of proper smooth families of finite presentation $X \to T$ whose geometric fibres are Del Pezzo surfaces of degree $2$ and morphisms are isomorphisms over $T$. \end{definition}
 
\begin{proposition} \label{prop:delpezzos} There is a natural equivalence  $\mathcal{D}_{2, \mathbb{Z}[1/2]} \simeq \mathcal{H}^{\text{sm}}(2,2,2)_{\mathbb{Z}[1/2]}$.
\end{proposition}

\begin{proof} This is well-known so we give a sketch of the argument. Given a proper smooth family of Del Pezzo's of degree two $f: X \to T$, cohomology and base change implies that there exists a double cover $\pi: X \to \mathbb{P}(f_*\omega_{X/T}^{\vee}) \to T$ which is ramified along a family of quartic curves. Moreover, $\pi$ has the structure of a relative uniform cyclic cover in a unique way (see Lemma \ref{lem:doublecovers}). In particular, a morphism of objects in $\mathcal{D}_2$ induces an equivariant morphism of relative uniform cyclic covers. Thus, there is a functor $F: \mathcal{D}_2 \to \mathcal{H}^{\text{sm}}(2,2,2)$. On the other hand, the forgetful morphism $G: \mathcal{H}^{\text{sm}}(2,2,2) \to \mathcal{D}_2$ which sends
\[X \to P \to S \quad \mapsto \quad X \to S\]
is inverse to $F$. That $G \circ F=\text{id}$ is clear and  $F \circ G \simeq \text{id}$ is left to the reader.
\end{proof}

Recall from Example \ref{ex:k3s} that $\mathcal{F}_{2}$ denotes the category of \emph{degree two polarized K3 surfaces}. 

\begin{proposition} \label{prop:k3s} There is a natural equivalence $\mathcal{F}_{2, \mathbb{Z}[1/2]} \cong \mathcal{H}^{\text{sm}}(2,2,3)_{\mathbb{Z}[1/2]}$. 
\end{proposition} 

\begin{proof} An object of $\mathcal{H}^{\text{sm}}(2,2,3)_{\mathbb{Z}[1/2]}(T)$ is of the form $X \to P \to T$ where each $X_t$ is a double cover branched along a sextic in $P_{\bar{t}} \simeq \mathbb{P}^2_{\kappa(\bar{t})}$. Thus, $X \to T$ is a proper smooth family of degree $2$ K3 surfaces, and defining $\lambda$ to be the pull-back of the ample generator in $\Pic_{P/S}$ to $\Pic_{X/S}$  yields a morphism $F: \mathcal{H}^{\text{sm}}(2,2,3)_{\mathbb{Z}[1/2]} \to \mathcal{F}_{2,\mathbb{Z}[1/2]}$. 

We next describe an inverse of the functor $F$. The universal object $(X \to \mathcal{F}_{2, \mathbb{Z}[1/2]}, \lambda)$ admits a morphism $X \to P \to \mathcal{F}_{2, \mathbb{Z}[1/2]}$ where $P$ is a relative Brauer-Severi scheme of dimension $2$ by Proposition \ref{prop:factorBS}. Since this is a finite flat cover of degree $2$, Lemma \ref{lem:doublecovers} yields a unique structure of a relative uniform cyclic cover. Thus, we obtain a map $G: \mathcal{F}_{2, \mathbb{Z}[1/2]} \to \mathcal{H}^{\text{sm}}(2,2,3)_{\mathbb{Z}[1/2]}$ which is inverse to $F$ (again by Proposition  \ref{prop:factorBS}).  
\end{proof}

From Theorem \ref{thm:H2rd} and Propositions \ref{prop:delpezzos}, \ref{prop:k3s},  we obtain the following Corollary. This recovers special cases of results of Scholl \cite{Scholl} and Andr\'e \cite{Andre}, and generalises Scholl's result from number fields to finitely generated fields of characteristic $0$.

\begin{corollary} If $A$ is an integrally closed $\mathbb{Z}$-finitely generated integral domain with $\chr \kappa(A) = 0$ and $2\in A^\times$, then the sets
	$\pi_0(\mathcal{D}_2(A))$ and $\pi_0(\mathcal{F}_2(A))$ are  finite.
\end{corollary}

\subsection{Proof of Theorem \ref{thm:weighted}}
Let $r,A$ be as in Theorem \ref{thm:weighted} and let $\mathcal{X} \subset \P(1,1,1,r)_A$ be a smooth surface of degree $2r$. Then $\mathcal{X}$ has the natural structure of a double cover of $\P^2_A$ given by projecting to the first three coordinates and using Lemma \ref{lem:doublecovers}. Thus by Theorem \ref{thm:H2rd}, there are finitely many possibilities for $\mathcal{X}$ up to isomorphism,  as a double cover. Hence  only finitely many possibilities up to abstract isomorphism. \qed

\section{Abelian hypersurfaces and cyclic covers}

In this section we show that stacks of cyclic covers of  hypersurfaces in an abelian variety are arithmetically hyperbolic. To do this, we combine our cyclic covering method with recent work of Lawrence--Sawin \cite{lawrence2020shafarevich}.   For an application of our work here, we refer the reader to \cite{JM}.

\subsection{Abelian hypersurfaces}

Recall that a morphism $\pi: A \to S$ is said to be an \emph{abelian scheme} if it is a proper smooth group scheme with geometrically connected fibres. Moreover, for any torsor under an abelian scheme $P/S$ the functor $\Pic_{P/S}^{0}$ is representable by an abelian scheme over $S$, and if $P=A$, we denote this by $A^{\vee}$ and refer to it as the \emph{dual} of $A$ (see \cite[Remark 1.5]{faltingschai} and \cite[Prop.~2.1.3]{abelianolsson}). A \emph{degree d polarization} on $A$ is a finite flat map $\lambda: A \to A^{\vee}$ of group schemes whose kernel is finite locally free over $S$ of degree $d^2$. Given a relatively ample line bundle $L$ on an abelian scheme $A/S$, we say \emph{L has degree d} if the map
\[\lambda_L: A \to A^{\vee}, \quad a \mapsto t_a^*L \otimes L^{\vee},\]
is a degree $d$ polarization. Moreover, given a relatively ample line bundle $L$ on a torsor $\pi: P \to S$ under an abelian scheme $A/S$, the sheaf $\pi_*L$ is locally free of some rank $d$ and $L$ induces a polarization $\lambda_{L}: A \to A^{\vee}$ of degree $d$ (see \cite[2.2.3 and 2.2.4]{olssonabelian2}).

\begin{definition} Let $\mathcal{AH}_{g,d}$ denote the fibred category over the category of schemes, whose fibre over a scheme $S$ is the groupoid of triples $(\pi: P \to S, L, s:\mathcal{O}_P \to L)$ such that 
\begin{enumerate} 
\item $\pi: P \to S$ is proper, flat morphism of finite presentation, and each geometric fibre $P_{\bar{t}}$ admits the structure of an abelian variety of dimension $g$.
\item $L \in \Pic P$ is a line bundle such that $L|_{P_{\bar{t}}}$ is an ample line bundle of degree  $d$ on every geometric fibre $P_{\bar{t}}$.
\item the zero locus $V(s) \subset P$ is flat over $S$. 
\end{enumerate} 
A morphism $(\pi': P' \to S', L', s':\mathcal{O}_{P'} \to L') \to (\pi: P \to S, L, s:\mathcal{O}_{P} \to L)$ over $S' \to S$ consists of a morphism $f$ which makes the following square Cartesian

\begin{center}
\begin{tikzcd}
 P' \arrow[d] \arrow[r, "f"] & P \arrow[d]  \\
 S' \arrow[r] & S 
\end{tikzcd}
\end{center}
and an isomorphism $g: f^*L \simeq L'$ which sends $f^*s$ to $s'$. We call $\mathcal{AH}_{g,d}$ the \emph{moduli stack of abelian hypersurfaces of degree d}. The subcategory $\mathcal{AH}_{g,d}^{\text{sm}} \subset \mathcal{AH}_{g,d}$ consisting of $(\pi: P \to S, L, s:\mathcal{O}_P \to L)$ where $V(s) \subset P$ is smooth over $S$ will be referred to as the \emph{moduli stack of smooth abelian hypersurfaces of degree d}.
\end{definition}

\begin{remark} \label{rem:clarifyH}
The locally principal subscheme $V(s) \subset P$ is flat over $S$ if and only if it is Cartier on every fibre (see \cite[Tag 062Y]{stacks-project}). In other words, $V(s)$ is flat over $S$ if and only if $f_*s: \mathcal{O}_S \to \pi_*L$ avoids the zero section.
\end{remark} 

First we show these are stacks of divisorial pairs, in the sense of Definition \ref{def:pairs}.

\begin{proposition} \label{prop:Hpairs} The fibred categories $\mathcal{AH}_{g,d}$ and $\mathcal{AH}_{g,d}^{\text{sm}}$ are moduli stacks of divisorial pairs which have affine diagonal.
\end{proposition}

\begin{proof} 
We first construct the associated stack with polarising line bundle. 
Consider the fibred category $\overline{\mathcal{AH}}_{g,d}$ whose fibre consists of $(\pi: P \to S, L)$ where 
\begin{enumerate} \item $\pi: P \to S$ is a proper, flat morphism of finite presentation and each geometric fibre $P_{\bar{s}}$ admits the structure of an abelian variety of dimension $g$.
\item $L \in \Pic P$ is a line bundle such that $L|_{P_{\bar{s}}}$ is a degree $d$ ample line bundle for each geometric fibre $P_{\bar{s}}$.
\end{enumerate}
A morphism $(\pi: P' \to S', L') \to (\pi: P \to S, L)$ is given by a morphism $f$ making the diagram Cartesian  

\begin{center}
    \begin{tikzcd}
 P' \arrow[d] \arrow[r, "f"] & P \arrow[d]  \\
 S' \arrow[r] & S 
\end{tikzcd}
\end{center}
and an isomorphism $g: f^*L \simeq L'$. 

We claim that this is an open substack of $\mathcal{P}ol^{\mathcal{L}}$. Indeed, the locus where the morphism $f: P \to T$ is smooth with geometrically integral fibres is open (see \cite[Appendix E.1]{torstenalg}) and the locus where $P_{\bar{t}}$ has the structure of an abelian variety is open in $T$. To see the latter claim: by standard approximation methods it suffices to assume that $T$ is Noetherian and we may also assume that $T$ is connected with a point $\bar{t}: \Spec k \to T$ such that $P_{\bar{t}}$ is an abelian variety. Base changing along $P \to T$ one obtains a section, thus we may apply \cite[Thm.~6.14]{GIT} to deduce that $f_P: P \times_T P \to P$ (equipped with the diagonal map) is an abelian scheme. 

Let $(\mathcal{U} \to \mathcal{P}ol^{\mathcal{L}}, \mathcal{L})$ denote the universal object of $\mathcal{P}ol^{\mathcal{L}}$, so far we have cut out an open locus $\mathcal{V} \subset \mathcal{P}ol^{\mathcal{L}}$ where the geometric fibres of the universal family $\mathcal{U}|_{\mathcal{V}} \to \mathcal{V}$ admit the structure of an abelian variety. By \cite[Prop.~6.13]{GIT} $\overline{\mathcal{AH}}_{g,d}$ is open in $\mathcal{V}$ since it defines the locus where the accompanying line bundle has degree $d$ on each fibre. Thus $\overline{\mathcal{AH}}_{g,d}$ is an open substack of $\mathcal{P}ol^{\mathcal{L}}$, and it has affine diagonal by \cite[\S 2.1]{dJS10}.

To conclude, observe that if $(f:\mathcal{P} \to \overline{\mathcal{AH}}_{g,d}, \mathcal{L})$ is the universal object of $\overline{\mathcal{AH}}_{g,d}$ then $\mathcal{AH}_{g,d}$ can be realized as the complement of the zero section in the tautological rank $d$ vector bundle $f_*(\mathcal{L})$ over $\overline{\mathcal{AH}}_{g,d}$. It follows that $\mathcal{AH}_{g,d}$ is a moduli stack of divisorial pairs, and the same follows for the open substack $\mathcal{AH}_{g,d}^{\text{sm}}$. Lastly, since the morphism $\mathcal{AH}_{g,d} \to \overline{\mathcal{AH}}_{g,d}$ is quasi-affine, the stack  $\mathcal{AH}_{g,d} $  has   affine diagonal. 
\end{proof}

\begin{proposition} \label{prop:Hseparated}
	The stack $\mathcal{AH}_{g,d}^{sm}$ has finite diagonal. In particular, it is separated. 
\end{proposition}
\begin{proof} 
The diagonal is proper by the Matsusaka-Mumford theorem \cite[Thm.~2]{MatMum} (see \cite[Thm.~4.3]{Popp} for a different formulation) and affine by Proposition \ref{prop:Hpairs}. Therefore,  the diagonal is finite, as required. 
\end{proof}
  
We briefly recall the Albanese torsor and its universal property. (For references in much greater generality see \cite[Thm.~VI.3.3]{FGA} and \cite[5.14, 5.20, and 5.21]{KlePic}.) Let $X$ be a smooth proper $S$-scheme where $S$ is defined over $\mathbb{Q}$ and assume that $X/S$ admits a relatively ample line bundle. Then there is an abelian $S$-scheme $\text{Alb}^0_{X/S}$, a torsor $\text{Alb}^1_{X/S}$ under it, and a $S$-morphism $X \to \text{Alb}^1_{X/S}$. This has the following universal property: if $P$ is a torsor under an abelian scheme $A/S$ and $X \to P$ is an $S$-morphism, then there is a unique factorization $X \to \text{Alb}^1_{X/S} \to P$ over $S$ and a unique map of abelian $S$-schemes $\text{Alb}^0_{X/S} \to A$ making  $\text{Alb}^1_{X/S} \to P$ equivariant.

The following is an application of the main result of Lawrence--Sawin \cite{lawrence2020shafarevich}.

\begin{theorem} \label{thm:sawinhyp} 
	Let $d \geq 1$ and $g \geq 4$ or $g=2$.
	Let $K$ be a number field and $S$ a finite set of finite places of $K$.
	Then $\pi_0(\mathcal{AH}_{g,d}^{sm}(\O_{K,S}))$ is finite.
\end{theorem}

\begin{proof} 
The result from \cite{lawrence2020shafarevich} applies to families of hypersurfaces in a fixed abelian variety. We want to show an analogous result for hypersurfaces in varying families of torsors under abelian varieties. To prove  this we use known finiteness statements and pass to a field extension to trivialise various data, then perform a descent back to the ground field.

We first show that $\mathcal{AH}_{g,d, \overline{\mathbb{Q}}}^{sm}$ is arithmetically hyperbolic. To do so, it suffices to show that for every number field $K \subset \overline{\mathbb{Q}}$ and every finite set of places $S$, the set
\begin{equation} \label{eqn:is_finite}
\text{Im}[\pi_0(\mathcal{AH}_{g,d}^{sm}(\O_{K,S})) \to \pi_0(\mathcal{AH}^{sm}_{g,d}(\overline{\mathbb{Q}}))] \quad \text{is finite}.
\end{equation} 
 Observe that for every $(X, L, s) \in \mathcal{AH}_{g,d}(\O_{K,S})$ the universal morphism to the Albanese torsor $X \to \text{Alb}^1_{X/\O_{K,S}}$ is an isomorphism. Thus, $X$ is naturally a torsor under the abelian scheme $\text{Alb}^0_{X/\O_{K,S}}$ over $\O_{K,S}$. By the Shafarevich conjecture for abelian varieties  \cite{Faltings2}, it suffices to assume that every $X$ is a torsor under a fixed abelian scheme $A \to \O_{K,S}$. 
Note that $L$ induces a degree $d$ polarization $\lambda: A \to A^{\vee}$ on $A$. But the set of isomorphism classes of polarizations of degree $d$ on $A$ is finite by \cite[Tem.~V.18.1]{CornellSilverman}, so we may assume that $\lambda_L$ induces a fixed degree $d$ polarization $\lambda$. In fact, we may also fix the isomorphism class of $X$ since the set, $J$, of isomorphism classes of $A$-torsors $X$ which admit a line bundle $L$ inducing the polarization $\lambda_L=\lambda$ is finite. Indeed, let $c_{\lambda}$ denote the image of $\lambda \in \text{NS}_{A}(\O_{K,S})$ in $\H^1(\Spec \O_{K,S}, A^{\vee})$ then by \cite[Prop.~4.3]{poonen1999cassels} the set $J$ is the inverse image of $c_{\lambda}$ along the map
\[\H^1(\lambda): \H^1(\O_{K,S}, A) \to \H^1(\O_{K,S}, A^{\vee})\]
Thus $J$ is a torsor under a quotient of $\H^1(\O_{K,S}, \text{Ker}(\lambda))$ and the latter is finite by Lemma \ref{lem:GMB}.

To prove \eqref{eqn:is_finite}, we base change to a larger number field $K \subset K'$ so that $X_{K'}$ admits a section and $X_{\O_{K',S'}}$ inherits the structure of an abelian scheme, where $S'$ denotes the set of places of $K'$ above $S$. Now the set of sections $s: \mathcal{O}_X \to L$, with smooth zero locus, up to isomorphism, is finite by \cite[Thm.~1.1]{lawrence2020shafarevich} ($g \geq 4$) and Faltings \cite{Faltings2} ($g=2$). This gives the required finiteness, hence proves that $\mathcal{AH}_{g,d, \overline{\mathbb{Q}}}^{sm}$ is arithmetically hyperbolic.
To complete the proof, by Proposition \ref{prop:Hseparated} we know that $\mathcal{AH}_{g,d}^{sm}$ has finite diagonal. Therefore,  since $\mathcal{AH}_{g,d, \overline{\mathbb{Q}}}^{sm}$ is arithmetically hyperbolic over $\overline{\mathbb{Q}}$, the result   follows from Theorem \ref{thm:twists}. 
\end{proof}

\subsection{Double covers}

Theorem \ref{thm:sawinhyp} and Corollary \ref{cor:finite_dp} immediately give finiteness results for integral points on cyclic covering stacks of stacks of polarised abelian varieties. We make these explicit for double covers to obtain infinitely many new moduli stacks of canonically polarized varieties which are arithmetically hyperbolic. 

\begin{definition} 
We define $\mathcal{G}_{g,p} \to \Spec \mathbb{Q}$ to be a category fibred in groupoids whose fibre category over a scheme $S$ consists of morphisms $\pi: X \to S$ such that

\begin{enumerate} \item $\pi$ is a smooth proper morphism  of relative dimension $g$
\item the  fibres of $\pi$ are connected varieties of general-type with geometric genus $p$.
\item the Albanese map of the geometric fibres of $\pi$ is finite flat of degree $2$.
\end{enumerate}

\noindent Morphisms in this fibre category are isomorphisms over $S$. We call $\mathcal{G}_{g,p}$ the stack of \emph{double Albanese varieties} of dimension $g$ and genus $p$. 
\end{definition}

Such varieties occur for example amongst surfaces $S$ of general type with $q(S) = 2$ and $K_S^2 = 4\chi(\O_S)$ \cite[Thm.~0.1]{Man03}. 

\begin{proposition} \label{prop:G_{g,p}}
 The category $\mathcal{G}_{g,p}$ is a moduli stack of polarized varieties. 
 \end{proposition}

\begin{proof} 
We first prove that each $X \in \mathcal{G}_{g,p}(k)$ is canonically polarized. To show this, we may assume that $k$ is algebraically closed, since the canonical bundle is defined over the ground field. Then, by definition and Lemma \ref{lem:doublecovers}, there is an abelian variety $A$ over $k$ and a uniform cyclic cover $f: X \to A$ of degree $2$. So if we write $X=\Spec_A \mathcal{O}_A \oplus L$, then $\omega_{X/k}=f^*L^{-1}$ (see e.g. \cite[Prop. 3.5]{Wehler}) because $\omega_{A/k}=\mathcal{O}_A$.  Since $\omega_{X/k}$ is big, the line bundle  $L^{-1}$ is big as well. The claim then follows from the fact that any big line bundle on an abelian variety is ample (use \cite[Ex.~1.4.7, Cor.~1.5.18, and Thm.~2.2.16]{Laz04}).

Next consider the universal object $(u: \mathcal{U} \to \mathcal{P}ol^\Lambda,\lambda)$ and note that there is a locally closed subscheme $\mathcal{P}ol^{\omega} \subset \mathcal{P}ol^\Lambda$ where $\lambda|_{\mathcal{P}ol^{\omega}}$ represents the class of $\omega=\omega_{\mathcal{U}/\mathcal{P}ol^\Lambda}|_{\mathcal{P}ol^{\omega}}$ and is universal with respect to this property (see \cite[Cor.~4.3.2]{abramovich2011stable}). We can further restrict to the locally closed locus where the fibres of $u$ have dimension $g$ and $\pi_*\omega$ is locally free of rank $p$; call this stack $\mathcal{J}$. Finally, consider the map to the universal Albanese torsor $a: \mathcal{U}|_{\mathcal{J}} \to \text{Alb}^1_{\mathcal{U}/\mathcal{J}}$ (which exists by \cite[5.20, 5.21]{KlePic} and \cite[Thm.~3.3]{FGA})  
and observe that the locus where it is finite flat of degree $2$ is open in $\mathcal{J}$ (see \cite[Appendix E]{torstenalg}). Since isomorphisms preserve the canonical line bundle, this locally closed substack of $\mathcal{P}ol^{\omega}$ is equivalent to $\mathcal{G}_{g,p}$.
\end{proof}

Note that the universal object $\mathcal{U} \to \mathcal{G}_{g,p}$ is a relative uniform cyclic cover over its Albanese torsor since it is  finite locally free of degree $2$ (see \ref{lem:doublecovers}):
\[a: \mathcal{U} \to \text{Alb}^1_{\mathcal{U}/\mathcal{G}_{g,p}}\]
Taking the branch locus of $a$ yields a pair $(\text{Alb}^1_{\mathcal{U}/\mathcal{G}_{g,p}} \to \mathcal{G}_{g,p}, H)$ of a torsor under an abelian scheme and a divisor. By the following lemma and Proposition \ref{rem:smoothness}, this yields a natural morphism $\mathcal{G}_{g,p} \to \mathcal{AH}_{g,2^g(p-1), \mathbb{Q}}$ from the universal property of $\mathcal{AH}_{g,2^g(p-1), \mathbb{Q}}$.

\begin{lemma}
	$H$ is a relatively ample divisor of degree $2^{g}(p-1)$.
\end{lemma}
\begin{proof}
To prove the result, we may work over an algebraically closed field of characteristic $0$ (see \cite[Thm.~1.7.8]{Laz04}). 
Recall from the proof of Proposition \ref{prop:G_{g,p}} that $f: X=\Spec_A \mathcal{O}_A \oplus L \to A$ for some line bundle $L$ for which $L^{-1}$ is ample.  Recall that $X$ has dimension $g$ and geometric genus $p=\mathrm{h}^0(X,\omega_X)$. Let $H \subset A$ denote the branch locus of $f$ and note that $\mathcal{O}(H)=L^{-2}$. Then $\omega_X=f^*(\omega_A \otimes L^{-1})=f^*L^{-1}$ because $\omega_A=\mathcal{O}_A$. Therefore, by pushing forward and using the push-pull formula, we have
\[p=\mathrm{h}^0(\omega_X)=\mathrm{h}^0(f_*\omega_X)= \mathrm{h}^0(L^{-1} \otimes (\mathcal{O}_A \oplus L))= \mathrm{h}^0(\mathcal{O}_A) +\mathrm{h}^0(L^{-1})=1+\mathrm{h}^0(L^{-1}).\]
 
This gives $h^0(L^{-1})=p-1$. Now, we want to calculate the degree of $\mathcal{O}(H)=L^{-2}$. By \cite[Thm.~V.13.3]{CornellSilverman} this is the same as $\chi(L^{-2})$. Recalling that ample line bundles on abelian varieties have vanishing higher cohomology \cite[Prop.~6.13]{GIT}, the Riemann-Roch theorem \cite[Thm.~V.13.3]{CornellSilverman} gives
\[\chi(L^{-2})=(L^{-2})^g/g!=2^g(L^{-1})^g/g!=2^gh^0(L^{-1})\]
(here $(M)^g$ denotes the top self-intersection of the line bundle $M$). Putting it all together, we have that the degree of the branch locus of $f$ is $2^{g}(p-1)$.
\end{proof} 

\begin{proposition} \label{prop:alb2finiteproper} The morphism $\mathcal{G}_{g,p} \to \mathcal{AH}_{g,2^g(p-1), \mathbb{Q}}$ is proper and  \'etale. \end{proposition}

\begin{proof} Since $a$ realizes $\mathcal{U}$ as a relative uniform cyclic cover of degree $2$ over a torsor under an abelian scheme, we obtain a morphism $\mathcal{G}_{g,p}$ to the stack $\mathcal{X}_2$ of degree $2$ cyclic covers associated to the moduli stack of pairs $\mathcal{AH}_{g,2^g(p-1)}$ (see Theorem \ref{thm:properetale}). We claim that the functor
\[F: \mathcal{G}_{g,p} \to \mathcal{X}_2, \quad (X \to S) \, \mapsto \,  (X \to \text{Alb}^1_{X/S} \to S) \] 
is an equivalence of categories, with inverse given by the forgetful functor
\[G: \mathcal{X}_2 \to \mathcal{G}_{g,p}, \quad (X \to P \to S) \, \mapsto \, (X \to S).\]
It is clear that $G \circ F=\text{id}$ and note that the universal property of the Albanese torsor yields a natural transformation $F \circ G \to \text{id}$. This is an isomorphism because, if $X \to P \to S$ is a relative uniform cyclic cover of degree $2$, then $P$ is isomorphic to the Albanese torsor of $X/S$. Indeed, there is a factorization $X \to \text{Alb}^1_{X/S} \to P$ over $S$ and since $X \to \text{Alb}^1_{X/S}$ is finite but not an isomorphism, and $X \to P$ is a double cover, it follows that $\text{Alb}^1_{X/S} \to P$ is an isomorphism. Since $\mathcal{X}_2$ is proper and \'etale over $\mathcal{AH}_{g,2^g(p-1)}$ by Theorem \ref{thm:properetale}, this concludes the proof.
\end{proof}

We finally prove the following more precise version of Theorem \ref{mainthm:gentype}.

\begin{theorem} \label{thm:Ggp}
	Let $p,g \in \N$ with $g=2$ or $g \geq 4$. There exists a finite set of primes $T$
	and a model $\mathcal{G}_{g,p,\Z[T^{-1}]}$ for $\mathcal{G}_{g,p}$ over $\Z[T^{-1}]$
	with the following property.
	
	Let $K$ be a number field and $S$ a finite set of finite places of $K$ containing all places
	above $T$. Then $\pi_0(\mathcal{G}_{g,p,\Z[T^{-1}]}(\O_{K,S}))$ is finite.
\end{theorem}
\begin{proof}
	By Proposition \ref{prop:alb2finiteproper} and
	spreading out \cite[Prop.~B.3]{Rydh2}, there exists a model 
	$\mathcal{G}_{g,p,\Z[T^{-1}]}$ for $\mathcal{G}_{g,p}$ over $\Z[T^{-1}]$ for some
	finite set of primes $T$ together with a morphism 
	$\mathcal{G}_{g,p,\Z[T^{-1}]} \to \mathcal{AH}_{g,2^g(p-1), \Z[T^{-1}]}$ which is proper
	\'etale. Increasing $T$ again if necessary, by Proposition \ref{prop:Hseparated}
	we may assume that $\mathcal{AH}_{g,2^g(p-1), \Z[T^{-1}]}$ has finite diagonal, hence
	so does $\mathcal{G}_{g,p,\Z[T^{-1}]}$. The result now follows from Lemma \ref{lem:qf}
	and  Theorems \ref{thm:twists} and \ref{thm:sawinhyp}.
\end{proof}

	\bibliography{refsci}{}
	\bibliographystyle{plain}
	
\end{document}